\documentclass[10pt]{article}
\usepackage{amsthm,amsfonts,amsmath,amscd,amssymb,epsfig}
\usepackage{enumerate}

\title
{Compactness for holomorphic curves with switching Lagrangian
  boundary conditions}
\author{K.~Cieliebak, T.~Ekholm and J.~Latschev}
\date{12 March 2009}
\theoremstyle{plain}
\newtheorem{theorem}{Theorem}[section]
\newtheorem{thm}[theorem]{Theorem}
\newtheorem{corollary}[theorem]{Corollary}

\newtheorem{proposition}[theorem]{Proposition}
\newtheorem{prop}[theorem]{Proposition}
\newtheorem{lemma}[theorem]{Lemma}

\theoremstyle{remark}

\newtheorem*{remark}{Remark}
\newtheorem{example}[theorem]{Example}

\theoremstyle{definition}

\newcommand{\id}{{{\mathchoice {\rm 1\mskip-4mu l} {\rm 1\mskip-4mu l}
{\rm 1\mskip-4.5mu l} {\rm 1\mskip-5mu l}}}}

\newcommand{\p}{\partial}

\newcommand{\om}{\omega}

\newcommand{\eps}{\varepsilon}
\newcommand{\pHi}{\varphi}
\newcommand{\into}{\hookrightarrow}
\newcommand{\la}{\langle}
\newcommand{\ra}{\rangle}
\newcommand{\N}{{\mathbb{N}}}

\newcommand{\R}{{\mathbb{R}}}
\renewcommand{\H}{{\mathbb{H}}}
\newcommand{\C}{{\mathbb{C}}}

\newcommand{\ind}{{\rm ind}}
\newcommand{\st}{{\rm st}}

\renewcommand{\min}{{\rm min}}
\renewcommand{\max}{{\rm max}}

\newcommand{\inn}{{\rm int\,}}

\newcommand{\loc}{{\rm loc}}

\newcommand{\diam}{{\rm diam\,}}

\newcommand{\MM}{\mathcal{M}}
\newcommand{\CC}{\mathcal{C}}

\newcommand{\UU}{\mathcal{U}}

\newcommand{\NN}{\mathcal{N}}
\newcommand{\TT}{\mathcal{T}}
\renewcommand{\Im}{{\mathfrak I \mathfrak m }}
\newcommand{\comment}[1]{}
\newcommand{\x}{\times}

\newcommand{\pa}{\partial}

\begin{document}
\maketitle
\abstract{
We prove a compactness result for holomorphic curves
with boundary on an immersed Lagrangian submanifold
with clean self-intersection. As a
consequence, we show that the number of intersections of such
holomorphic curves with the self-intersection locus is uniformly
bounded in terms of the Hofer energy.
}
%\tableofcontents

%%%%%%%%%%%%%%%%%%%%%%%%%%%%%%%%%%%%%%%%%%%%%%%%%%%%%%%%%%%%%%%%%%%%%
\section{Introduction}\label{sec:intro}
%%%%%%%%%%%%%%%%%%%%%%%%%%%%%%%%%%%%%%%%%%%%%%%%%%%%%%%%%%%%%%%%%%%%%

In this paper we prove a compactness result for holomorphic curves
with boundary on an immersed Lagrangian submanifold
with clean self-intersection along a compact submanifold $K$. As a
consequence, we show that the number of intersections of such
holomorphic curves with $K$ is uniformly bounded in terms of the Hofer
energy. This finiteness result is an essential ingredient in
the proof in~\cite{CELN} of the isomorphism of degree 0 Legendrian contact
homology of the unit conormal bundle of a knot $K\subset\R^3$ with the
cord algebra defined in~\cite{Ng:08}.

Consider a symplectic manifold $(X,\om)$ and an immersed Lagrangian
submanifold $L\subset X$ with clean self-intersection along a
compact submanifold $K$. Let $J$ be an $\om$-compatible almost complex
structure on $X$. We assume that near $K$ the structure $J$ is
integrable and $L$ is real analytic.
Let $(S,j)$ be a connected Riemmann surface with boundary $\p S$. A
{\em holomorphic curve} $f:(S,\p S,j)\to (X,L,J)$ is a continuous map
$f:S\to X$ which maps $\p S$ to $L$ and is $(j,J)$-holomorphic in the
interior. We allow $(X,L,\om,J)$ to be noncompact with cylindrical
ends as in~\cite{EGH}, and $S$ to have punctures in the interior as
well as on the boundary (see Section~\ref{sec:global} for the precise
setup). However, we do not treat intersections of $f|_{\p S}$ with $K$
-- which we call {\em switches} -- as boundary punctures. In
particular, we do not impose any constraints on the number and types
of switches.

Our first result states that the compactness result in symplectic
field theory (\cite{BEHWZ,CM-comp}) carries over to this setting.
See Section~\ref{sec:compact} for the precise statement.

\begin{thm}\label{thm:comp}
Under suitable hypotheses on $(X,L,\om,J)$ each sequence of
holomorphic curves $f_n:(S_n,\p S_n,j_n)\to (X,L,J)$ of fixed signature
and uniformly bounded energy has a subsequence converging in the sense
of~\cite{BEHWZ} to a stable holomorphic curve.
\end{thm}

As a consequence, we obtain the following finiteness result for the
number of switches.

\begin{thm}\label{thm:fin}
 In the situation of Theorem~\ref{thm:comp}, suppose in addition
 that\linebreak
$(X,L,\omega=d\lambda)$
is exact with convex end. Then for each
 $s\in \N$ and $C>0$ there exists a constant $\kappa(s,C)\in\N$ such
 that every holomorphic {\em disk} $f:(\dot D,\p\dot D,j)\to (X,L,J)$
 with at most $s$ boundary punctures and energy $\leq C$ has at most
 $\kappa(s,C)$ switches.
\end{thm}

The case $s=1$ of this finiteness result is an essential ingredient in
the proof in~\cite{CELN} of the isomorphism of degree 0 Legendrian contact
homology of the unit conormal bundle of a knot $K\subset\R^3$ with the
cord algebra defined in \cite{Ng:08}. This isomorphism is constructed
by counting 1-punctured holomorphic disks in $T^*\R^3$ with boundary on
the immersed Lagrangian submanifold $L=NK\cup\R^3$, where the conormal
bundle $NK\subset T^*\R^3$ of $K$ and the zero section $\R^3$
intersect cleanly along the knot $K$.

Holomorphic disks with boundary on cleanly intersecting Lagrangian
submanifolds are also studied in~\cite{AS}. 

\medskip

To put Theorem~\ref{thm:fin} into context, recall that in general
energy bounds are not enough to provide bounds on the topology of
holomorphic curves. Indeed, double branched covers of $\C P^1$ exist
for all genera, and by choosing the branch points to lie on the
equator and cutting the domain along preimages of suitable segments
connecting adjacent branch points, one obtains existence of
holomorphic curves of genus zero and arbitrarily many boundary
components, but of fixed energy.

Often, one can use index arguments to show that such phenomena
disappear after suitable perturbation.
%However, this does not always work.
Indeed, for the Fredholm theory of holomorphic curves
$f\colon(S,\pa S)\to (X,L)$ as above, it is convenient to puncture the
source $S$ at points in $\pa S$ that map to the clean intersection and
call such punctures {\em Lagrangian intersection punctures}. It turns
out that to each such puncture one can associate a winding number
$w\in \frac 1 2 \N$, and that the contribution of a Lagrangian
intersection puncture in a clean intersection of codimension $d$ to
the Fredholm index is $1-wd$ (see the appendix for more
details). Consequently, this contribution is negative provided $d\geq
3$, and equal to $0$ when $d=2$ and $w=\frac12$. It follows that for clean
intersections of codimension at least three one can control the number
of switches using transversality arguments. However, for codimension
two -- which is the most interesting case from the point of view of
smooth embedding theory~\cite{CELN} --
%in view of conormal lifts of codimension two submanifolds
no such argument is available. Still, the result of this
paper provides a bound on the number of switches which is independent
of codimension.

Similar remarks apply to the number of boundary circles $r$ and the
genus $g$ of $S$: If $\dim(L)=n$ satisfies $n>3$ then the number of
boundary circles and the genus can be bounded using transversality
arguments; this again follows from the dimension formula for the
corresponding moduli spaces (see the appendix). However, if $n=3$ the
dimension is independent of $g$ and $r$ and no such argument is
available. Indeed, the contribution to the Gromov-Witten invariant of
a Calabi-Yau 3-fold of multiple covers of degree $d$ and genus $g$ of
a fixed rational curve has been computed in~\cite{FP}; it is
nontrivial for any fixed $d\geq 2$ and arbitrarily high genus $g$, so
there is no bound of the genus in terms of the degree. It would be
interesting to have similar formulae for multiple covers of genus
zero and many boundary components of a fixed (punctured) disk.
\medskip

Our method of proof uses the integrability of $J$ near $K$ in an
essential way. It would be interesting to understand to what extend 
the conclusion of Theorem 1.1 remains true for more general almost
complex structures.

%%%%%%%%%%%%%%%%%%%%%%%%%%%%%%%%%%%%%%%%%%%%%%%%%%%%%%%%%%%%%%%%%%%%%
\section{Local theory}\label{sec:local}
%%%%%%%%%%%%%%%%%%%%%%%%%%%%%%%%%%%%%%%%%%%%%%%%%%%%%%%%%%%%%%%%%%%%%

Let $(S,j)$ be a connected Riemmann surface with boundary $\p S$,
possibly noncompact. We will consider functions $f:S \to \C$
satisfying the following conditions:
\begin{enumerate}[({F}1)]
\item $f$ is continuous on $S$;
\item $f$ is holomorphic on $\inn S := S\setminus\p S$;
\item $f$ maps $\p S$ to $\R \cup i\R$.
\end{enumerate}

We start with some elementary observations.
\begin{lemma}\label{lem:1}
A function $f$ satisfying (F1-3) is holomorphic on $\inn S \cup
(\p S \setminus f^{-1}(0))$.
\end{lemma}

\begin{proof}
The function $g:=f^2:S \to \C$ is continuous on $S$, holomorphic
on $\inn S$ and maps $\p S$ to $\R$, so by the Schwarz
reflection principle it is holomorphic on all of $S$. Since the
square root has holomorphic branches outside zero, the result for $f$
follows.
\end{proof}

\begin{lemma}\label{lem:2}
If $f:S \to \C$ satisfying (F1-3) is not identically zero, then it has
only finitely many zeroes in any compact subdomain $S' \subset S$.
\end{lemma}

\begin{proof}

If not, then $g=f^2$ is a holomorphic function for which $g^{-1}(0)$
has a limit point, forcing it to vanish identically.
\end{proof}

For $f:S \to \C$ satisfying (F1-3) and not identically zero, let
$\gamma$ be a path in $S$ which does not meet any zero of $f$. Define
the {\em winding number} of $f$ along $\gamma$ by
$$
   w(f,\gamma):=\frac{1}{2\pi}\int_\gamma f^*d\theta,
$$
where $d\theta$ denotes the angular form on $\C\setminus\{0\}$.

\begin{lemma}\label{lem:wind}
Suppose $f:S \to \C$ satisfies (F1-3) and is not identically zero. Let
$S'\subset S$ be a compact subset with piecewise smooth boundary $\p
S'=(S'\cap\p S)\cup\Gamma$, where $\Gamma$ is a union of disjoint arcs in $S$
not meeting any zero of $f$. Then
$$
   w(f,\Gamma) \geq \#(f^{-1}(0)\cap\inn S') +
   \frac{1}{4}\#(f^{-1}(0)\cap S'\cap\p S).
$$
\end{lemma}

\begin{proof}
Around each zero $p\in\inn S'$ pick a small disk $D_p\subset\inn S'$
containing no other zero. Then $w(f,\p D_p)=k\in\N$, where $(z-p)^k$
is the first nonvanishing term in the power series expansion of $f$ at
$p$.
Around each zero $q\in S'\cap\p S$ pick a small half-disk $D_q^+\subset
S'\setminus(\p S'\cap\inn S)$ containing no other zero and set
$\p^+D_q^+:=\p D_q\setminus\p S$. Then $w(f,\p^+D_q^+)=k/4\in\N/4$, where
$(z-q)^k$ is the first nonvanishing term in the power series expansion
of the holomorphic function $g=f^2$ at $q$. Now let $S''$ be the
region obtained by removing from $S'$ all disks resp.~half-disks
around zeroes of $f$. Since $d\theta$ is closed and the angle
$f^*\theta$ is constant along parts of $\p S$ containing no zeroes,
Stokes' theorem yields
$$
   0 = w(f,\p S'') = w(f,\Gamma) - \sum_p w(f,\p D_p) -
   \sum_q(f,\p^+D_q^+),
$$
from which the lemma follows.
%Consider the holomorphic function $g=f^2$ and note that $\int_\Gamma
%g^*d\theta=2\int_\Gamma f^*d\theta$.
\end{proof}

\begin{lemma}\label{lem:3}
Let $f_n:S \to \C$ be a sequence of functions satisfying (F1-3),
and assume that there is a constant $C>0$ such that for all $n \geq 1$
and all $z \in S$ we have
\begin{equation}\label{eq:bounds0}
|f_n(z)| \leq C
\end{equation}
Then there exists a subsequence $f_{n'}$ of the $f_n$,
and a function $f:S\to \C$ satisfying (F1-3) such that
\begin{enumerate}[(i)]
%\item $f$ is continuous on $D^+$ and holomorphic on $\inn D^+ \cup (\p
%  D^+ \setminus f^{-1}(0))$,
\item $f_{n'} \to f$ in $C^0_\loc$ on $S$, and
\item $f_{n'} \to f$ in $C^\infty_\loc$ on $\inn S \cup (\p S
  \setminus f^{-1}(0))$.
\end{enumerate}
\end{lemma}
\begin{proof}
Consider the associated sequence of holomorphic functions
$g_n:=f_n^2:S\to \C$. The assumptions imply that for all $z\in S$ we
have
$$
   |g_n(z)| \leq C^2.
$$
Hence by Montel's theorem, after passing to a subsequence, the $g_n$
converge in $C^\infty_\loc(S)$ to a limit function $g:S \to \C$
which is holomorphic and maps $\p S$ to $\R$. By the same argument,
after passing to a further subsequence, the $f_n$ converge in
$C^\infty_\loc(\inn S)$ to a holomorphic function $f:\inn S \to \C$
satisfying $f^2=g|_{\inn S}$.

At points $z \in \p S$ with $g(z) \neq 0$ we extend $f$
by taking the branch of $\sqrt{g}$ that agrees with $f$ at interior
points near $z$, and at points $z \in \p S$ with $g(z)=0$ we set
$f(z):=0$. The resulting function $f:S\to \C$ satisfies (F1-3).
% and $f^2=g$ holds throughout $D^+$.
In particular, Lemma~\ref{lem:1} applies to show that $f$ is
holomorphic on $\inn S\cup(\p S \setminus f^{-1}(0))$.

$C^0_\loc$-convergence of the $f_n$ to $f$ follows from the
$C^0_\loc$-convergence of the $g_n$ to $g$ and continuity of the
square root. It remains to show $C^\infty_\loc$-convergence $f_n\to f$
on compact subsets of $\inn S\cup(\p S \setminus f^{-1}(0))$. If
$f\equiv 0$ this holds trivially, so suppose the $f$ does not
vanish identically. Fix a compact subset $S'\subset \inn S \cup (\p S
\setminus f^{-1}(0))$. By Lemma \ref{lem:2}, $f$ has only finitely
many zeroes in $S'$. Pick a compact subset $S_0\subset S'\cap\inn S$
containing all the zeroes and set $S_1:=S'\setminus\inn S_0$. On $S_0$
the $C^\infty$-convergence $f_n\to f$ was shown above, and on $S_1$ it
follows from the $C^\infty_\loc$-convergence $g_n\to g$ and smoothness
of the square root away from zero.
\end{proof}

The following statement is a variant of a result known as Vitali's
theorem.
\begin{lemma}\label{lem:4}
Let $f_n:S \to \C$ be a sequence of functions satisfying the
assumptions of Lemma~\ref{lem:3}, and suppose there exists a compact
subset $A\subset S$ such that each $f_n$ has at least $n$
zeroes in $A$.
%, counted with multiplicity.
Then the limiting function $f$ vanishes identically.
\end{lemma}

\begin{proof}
Pick a compact subset $S'\subset S$ with piecewise smooth boundary $\p
S'=(S'\cap\p S)\cup\Gamma$ such that $A\subset S'\setminus\Gamma$.
If $f$ has infinitely many zeroes in $S'$, then by Lemma~\ref{lem:2}
it vanishes. Otherwise, after passing to a subsequence, we may assume
that $f$ as well as each $f_n$ has only finitely many zeroes in
$S'$. After slightly shrinking $S'$ we may assume that $\Gamma$ avoids
the countably many zeroes of $f$ and the $f_n$. Since $A\subset
S'\setminus\Gamma$ and $f_n$ has at least $n$ zeroes in $A$,
Lemma~\ref{lem:wind} yields $w(f,\Gamma)\geq n/4$.
On the other hand, since $f_n|_\Gamma$ converges smoothly to
$f|_\Gamma$, we have
$$
   w(f_n,\Gamma) \stackrel{n \to\infty}\longrightarrow
   w(f,\Gamma)  < \infty,
$$
contradicting the previous estimate.
\end{proof}

%%%%%%%%%%%%%%%%%%%%%%%%%%%%%%%%%%%%%%%%%%%%%%%%%%%%%%%%%%%%%%%%%%%%%
\section{Global theory}\label{sec:global}
%%%%%%%%%%%%%%%%%%%%%%%%%%%%%%%%%%%%%%%%%%%%%%%%%%%%%%%%%%%%%%%%%%%%%

\subsection{Setup}\label{ss:setup}
%%%%%%%%%%%%%%%%%%%%%%%%%%%%%%%%%%%%%%%%%%%%%%%%%%%%%%%%%%%%%%%%%%%%%

For the global theory we consider the following setup.
\begin{enumerate}
\item[(X)] $(X,J)$ is an almost complex manifold with cylindrical
end $\R_+\times M$ adjusted to $(\om,\lambda)$ in the sense
of~\cite{BEHWZ}.
\end{enumerate}
This means that $X=\bar X\cup(\R_+\times M)$ with $\p\bar X=M$,
$(\om,\lambda)$ is a stable Hamitonian structure on $M$, $\om$ extends
to a symplectic form on $\bar X$, and $J$ is compatible with $\om$ on
$\bar X$ and with $(\om,\lambda)$ on $\R_+\times M$. We allow $X$ to
be noncompact but impose the following condition.
\begin{enumerate}
\item[(Y)] There exists a compact subset $\bar Y\subset\bar X$ such
that every $J$-holomorphic map $f:S\to X$ from a compact Riemann
surface with boundary satisfying $f(\p S)\subset Y:=\bar Y\cup(\R_+\times(\bar
Y\cap M))$ is entirely contained in $Y$.
\end{enumerate}
Note that condition (Y) is trivially satisfied (taking $\bar Y=\bar
X$) if $\bar X$ is compact.
Our assumption on the Lagrangian is the following.
\begin{enumerate}
\item[(L)] $L\subset Y \subset X$ is a properly
immersed Lagrangian submanifold with $L\cap(\R_+\times
M)=\R_+\times\Lambda$ for a compact submanifold $\Lambda\subset M$ satisfying
$\lambda|_\Lambda=\om|_\Lambda=0$, and such that $L$ has clean
self-intersection along a compact connected submanifold $K\subset\inn\bar Y$.
\end{enumerate}
Here {\em clean self-intersection} means that at each point $x\in K$
exactly two branches $L_0,L_1$ of $L$ meet and $T_xK=T_xL_0\cap
T_xL_1$. More precisely, $L$ is the image of a Lagrangian immersion
$f:\tilde L\to X$ with clean self-intersection along the submanifold
$\tilde K=f^{-1}(K)$. Then $f|_{\tilde K}:\tilde K\to K$ is a 2-1
covering and the two branches of $L$ near $x\in K$ are the images
under $f$ of neighbourhoods of the preimages $x_0,x_1$ of $x$. Note
that $L$ may be 2-sheeted near $K$, i.e.~the union of two embedded
submanifolds intersecting in $K$ (if $\tilde K$ is disconnected), or
1-sheeted (if $\tilde K$ is connected). We impose the following
condition on the almost complex structure near $K$.
%{\em $J$-orthgonality} means that for every $x\in K$
%the intersection $T_xL_0\cap J(T_xL_1)$ is $(n-k)$-dimensional, where
%$k=\dim K<\dim L$.
%By $\UU_\eps$ we denote the open $\eps$-neighbourhood of $K$
%in $X$ with respect to the distance induced by $(\om,J)$.

\begin{enumerate}
\item[(K)] There exists a neighbourhood $\UU$ of $K$ on which $J$
is integrable, a holomorphic embedding $\iota:K^\C\into X$ of a
complexification of $K$, and a holomorphic projection $\tau:\UU\to K^\C$
on a neighbourhood of $K$ such that $\tau\circ\iota=\id$. Moreover,
near every point $x\in K$ there exist holomorphic coordinates in
$\C^n=\R^k\oplus\R^{n-k}\oplus i\R^k\oplus i\R^{n-k}$
sending $x$ to $0$, $L_0$ to $\R^n$ and $L_1$ to $\R^k\oplus
i\R^{n-k}$.
\end{enumerate}

In particular, this implies that $L$ is real analytic near $K$ with
$J$-orthogonal self-intersection along $K$, i.e.~for every $x\in K$
the intersection $T_xL_0\cap J(T_xL_1)$ is
$(n-k)$-dimensional. However, condition (K) is more restrictive than
this. Indeed, not every pair of real analytic curves in $\C$
intersecting orthogonally at the origin can be mapped to the
coordinate axes by a local biholomorphism (e.g.~if one curve is the
$y$-axis, then the existence of such a biholomorphism imposes
infinitely many constraints on the Taylor coefficients of the other
curve as a graph over the $x$-axis).

Finally, we assume that the Reeb flow on $M$ satisfies the
following nondegeneracy condition:
\begin{enumerate}
\item[(R)] No closed Reeb orbit meets $\Lambda$, and all closed Reeb
  orbits and Reeb chords are non-degenerate.
\end{enumerate}
Here a {\em Reeb chord} is a Reeb orbit $\gamma:[0,T] \to M$ with
$\gamma(0),\gamma(T)\in \Lambda$. If there are no closed Reeb orbits
(e.g.~for conormal lifts of $K\subset\R^n$ with the flat metric)
these conditions can be arranged by a perturbation of $\Lambda$. In
the contact case $\om=d\lambda$
these conditions can be arranged by a perturbation of $\lambda$.

Our main case of interest is described in the following example.

\begin{example}[cotangent bundle]\label{ex:cot}
Here the symplectic manifold $X=T^*Q$ is the cotangent bundle of a
Riemannian manifold $Q$ with the Liouville 1-form $\lambda=p\,dq$ and
symplectic form $\om=d\lambda$. $M=S^*Q$ is the unit cotangent bundle
and $J$ is the almost complex structure on $T^*Q$ induced by the
Riemannian metric, deformed outside $S^*Q$ to make it cylindrical.
$K\subset Q$ is a compact submanifold and $L=Q\cup NK$, where
$Q$ is the zero section and $NK$ the conormal bundle, and
$\Lambda=NK\cap S^*Q$. Then $Q$ and $NK$ intersect cleanly along $K$.
%If we can make the metric flat near $K$ integrability of $J$ near
%$K$ follows and real analyticity of $L$ near $K$ can be achieved by
%making $K$ and the metric real analytic. (If this is not possible we
%use Lemma~\ref{lem:K2} below).
We assume that $Q=\bar Q\cup(\R_+\times\p\bar
Q)$ with compact $\bar Q$; then condition (Y) can be arranged (with
$Y=T^*\bar Q$) by making all level sets $\{r\}\times\p\bar Q$ in the
cylindrical end $\R_+\times\p\bar Q$ totally geodesic (then their
preimages in $T^*Q$ are Levi-flat and holomorphic curves cannot touch
them from inside). Condition (R) holds for a generic metric;
condition (K) can be arranged by Proposition~\ref{prop:K} below, or by
Remark~\ref{rem:K2} if $K$ admits a flat metric and has trivial normal
bundle (e.g.~for a 1-knot in $\R^3$).
%In the special case $Q=\R^n$ all conditions are satisfied if we take
%the flat metric on $\R^n$ and make $K$ real analytic (for $Y$ we can
%take the cotangent bundle of a ball containing $K$).
\end{example}

\subsection{Structure near $K$}\label{ss:K}
%%%%%%%%%%%%%%%%%%%%%%%%%%%%%%%%%%%%%%%%%%%%%%%%%%%%%%%%%%%%%%%%%%%%%

In this subsection we show that condition (K) can always be arranged
by a deformation of the compatible almost complex structure near $K$,
provided that $L$ is 2-sheeted near $K$.

\begin{proposition}\label{prop:K}
Let $L_0,L_1$ be Lagrangian submanifolds of a symplectic $2n$-manifold
$(X,\om)$ intersecting cleanly along a closed submanifold $K$ of
dimension $k$. Then there exists an $\om$-compatible integrable
complex structure $J$ on a neighbourhood $\UU$ of $K$ such that
condition (K) holds.
\end{proposition}

The proof of this proposition is based on three lemmata. The first one
provides a symplectic normal form for $L_0,L_1$ near $K$.

\begin{lemma}\label{lem:K1}
Let $L_0,L_1$ be Lagrangian submanifolds of a symplectic $2n$-manifold
$(X,\om)$ intersecting cleanly along a closed submanifold $K$ of
dimension $k$.
Then there exists a symplectomorphism from a neighbourhood $\UU$ of $K$
onto a neighbourhood of $K$ in $(T^*L_0,\om_\st)$ mapping $L_0$ to the
zero section and $L_1$ to the conormal bundle $NK$.
\end{lemma}

\begin{proof}
Consider the cotangent bundle $\pi:T^*L_0\to L_0$ with its standard
symplectic form $\om_\st$. By the Lagrangian neighbourhood theorem,
there exists a symplectomorphism from a neighbourhood $(\UU,\om)$ of
$K$ onto a neighbourhood of $K$ in $(T^*L_0,\om_\st)$ mapping $L_1$ to
the conormal
bundle $NK$. A short computation shows that the image $L_0'$ of $L_0$
under this symplectomorphism is tangent to the zero section $L_0$
along $K$. Thus after shrinking the neighbourhood we may assume that
$L_0'$ is the graph of a closed 1-form $\lambda$. Since $\lambda$
vanishes along $K$ it equals $dh$ for a function $h$ whose
differential vanishes along $K$.
The Hamiltonian flow of $h\circ\pi:T^*L_0\to\R$ is given by
$\phi_t(q,p)=(q,p+t\,d_qh)$. So the time-(-1)-map $\phi_{-1}$
preserves $NK$ and maps $L_0'$ to the zero section.
\end{proof}

Next we construct a holomorphic model for $L_0,L_1$ near $K$ for which
condition (K) holds.
Consider a complex vector bundle $E\to M$. A {\em holomorphic
  structure} on $E$ is given by the structure of complex manifolds on
$E$ and $M$ together with holomorphic local trivializations. By a {\em
  K\"ahler structure} on a holomorphic vector bundle $E$ we mean a
fibrewise linear K\"ahler form $\om_E$ on $E$.

\begin{lemma}\label{lem:K2}
Let $F\to K$ be a real vector bundle over a compact manifold $K$ and
$E\to TK$ the pullback of the
complexified bundle $F\otimes\C\to K$ to the tangent bundle $TK$. Then
there exists a K\"ahler vector bundle structure on $E$ for which the
total spaces of the subbundles $F\to K$ and $iF\to K$ are real
analytic, totally real and Lagrangian.
\end{lemma}

\begin{proof}
We first describe the real K\"ahler structures on the tautological
bundles over Grassmannians.
For positive integers $m<N$ consider the action of $GL(m,\C)$
on $\C^{m\times N}$ by left multiplication. We think of $\C^{m\times
N}$ as $m$-tuples of (row) vectors in $\C^N$ and denote by
$(\C^{m\times N})^*$ the subset of linearly independent tuples. The
maximal compact subgroup $U(m)\subset GL(m,\C)$ acts on $\C^{m\times
  N}$ in a Hamiltonian way (for the standard symplectic structure on
$\C^{m\times N}$) with moment map
$$
   \mu:\C^{m\times N}\to u(m),\qquad M\mapsto \frac{i}{2}XX^*.
$$
The quotient
$$
   G_\C:= G_\C(m,N) = (\C^{m\times N})^*/GL(m,\C) = \mu^{-1}(i/2)/U(m)
$$
is the Grassmannian of $m$-dimensional complex subspaces of $\C^N$; it
inherits the K\"ahler structure from $\C^{m\times N}$.

Next consider the set
$$
   V := \{(X,v)\in (\C^{m\times N})^*\times \C^N\mid v\in{\rm
   span}(X)\},
$$
where ${\rm span}(X)\subset\C^N$ denotes the complex subspace spanned by
the $m$-frame $X=(X_1,\dots,X_m$. Since the condition $v\in{\rm span}(X)$
can be expressed by complex equations -- the vanishing of all
$(m+1)$-dimensional minors of the matrix $(X_1,\dots,X_m,v)$ -- $V$ is
a complex submanifold of $\C^{m\times N}\times\C^N$. The quotient
$$
   \gamma_\C := V/GL(m,\C) \to G_\C,
$$
where $GL(m,\C)$ acts trivially on $v\in\C^N$, is the tautological
rank $m$ vector bundle over the Grassmannian $G_\C$. By construction,
it inherits from $\C^{m\times N}\times\C^N$ the structure of a
K\"ahler vector bundle.

Complex conjugation $\sigma(X,v):=(\bar X,\bar v)$ defines an
anti-holomorphic (i.e.~$\sigma\circ i=-i\circ\sigma$) and anti-symplectic
(i.e.~$\sigma^*\om_\st=-\om_\st$) involution of $\C^{m\times N}$. Since
$\sigma(UX,v)=\bar U\sigma(X,v)$ it descends to an anti-holomorphic and
anti-symplectic involution on $\gamma_\C$. Its fixed point
set, the total space of the tautological bundle
$$
   \gamma_\R \to G_\R:= G_\R(m,N)
$$
over the Grassmannian of real $m$-planes in $\R^N$, is therefore real
analytic, totally real and Lagrangian. The map $I(X,v):=(X,iv)$ on
$\C^{m\times N}\times\C^N$ is holomorphic and symplectic. Since it
commutes with the action of $GL(m,\C)$ and satisfies
$I\circ\sigma=-\sigma\circ I$, it descends to a holomorphic and
symplectic map on $\gamma_\C$ which anti-commutes with $\sigma$. Thus
the total space of the bundle
$$
   i\gamma_\R :=I(\gamma_\R) \to G_\R
$$
(whose fibre over a real subspace $W\subset\R^N$ is the subspace
$iW\subset\C^N$) is also real analytic, totally real and Lagrangian.

Now let $F\to K$ be a real vector bundle of rank $m$ over a compact manifold
$K$. Then for sufficiently large $N$ there exists a continuous map
$\phi:K\to G_\R$ such that $F\cong \phi^*\gamma_\R$. We equip $K$ with a
real analytic structure. Complexification yields a complex structure
on the total space of the tangent bundle $TK$ such that the zero
section is real analytic. We approximate $\phi$ by a real analytic
embedding into $G_\R$ (which is possible for $N$ large) and complexify
this to a holomorphic embedding $\phi_\C:TK\into G_\C$ (after
replacing $TK$ by a neighbourhood of the zero section and identifying
this again with $TK$). By construction, $\phi$ is covered by an
injective bundle map $\Phi:F\to \gamma_\R$. We complexify it to an
injective bundle map $F\otimes\C\to\gamma_\C$ mapping $iF$ to
$i\gamma_\R$ and extend it to an injective bundle map
$\Phi_\C:E\to\gamma_\C$ covering $\phi_\C$. Now the K\"ahler bundle
structure on $\gamma_\C|_{\phi_\C(TK)}$ pulls back under $\Phi_\C$ to
a K\"ahler bundle structure on $E\to TK$ with the desired properties.
\end{proof}

\begin{lemma}\label{lem:K3}
Let $E\to TK$ be as in Lemma~\ref{lem:K2} with $\dim K=k$ and ${\rm
rank\;} E=n-k$. Then hypothesis (K) is satisfied for $X=E$, $L_0=F$
and $L_1=iF$.
\end{lemma}

\begin{proof}
For the holomorphic vector bundle $\tau:E\to TK$, the projection
$\tau$ and the inclusion $\iota:TK\into E$ of the zero section are
holomorphic. Next consider $x\in K$. Pick a neighbourhood $V$ of $x$
in $K$ and a real analytic trivialization $\phi:F|_V\into
\R^k\times\R^{n-k}$ mapping $x$ to $0$. Complexify it to a holomorphic
embedding $\phi^\C:\NN\into\C^k\times\C^{n-k}$ of a neighbourhood $\NN$ of
$F|_V$ in $E$. By uniqueness of analytic continuation, the restriction
of $\phi^\C$ to each $E_x\cap\NN$ with $x\in V$ is complex linear and
we can extend it linearly to the whole fibre $E_x$. Thus we may assume
that $\NN$ contains $E|_V$. By construction, $\phi^\C$ maps $F|_V$ to
$\R^k\times\R^{n-k}$, and by complex linearity in the fibres it maps
$iF|_V$ to $\R^k\times i\R^{n-k}$.
\end{proof}

\begin{proof}[Proof of Proposition~\ref{prop:K}]
Let $(X,\om)$, $L_0$, $L_1$ and $K$ be as in the proposition. Let
$F\to K$ be the normal bundle of $K$ in $L_0$ and denote by $E\to TK$
the pullback bundle of $F\otimes\C\to K$ under the projection $TK\to
K$ as in Lemma~\ref{lem:K2}. Then a neighbourhood of $K$ in $X$ is
diffeomorphic to a neighbourhood of $K$ in $E$ such that $L_0$
corresponds to $F$ and $L_1$ to $iF$. Lemma~\ref{lem:K2} provides a
K\"ahler vector bundle structure on $E$, with K\"ahler form $\om_E$,
for which $F$ and $iF$ are Lagrangian. By Lemma~\ref{lem:K1}, the
quadruples $(X,\om,L_0,L_1)$ and $(E,\om_E,F,iF)$ are both isomorphic
near $K$ to the same standard model. Hence there exists a
symplectomorphism from a neighbourhood $\UU$ of $K$ in $(X,\om)$ to a
neighbourhood of $K$ in $(E,\om_E)$ mapping $L_0$ to $F$ and $L_1$ to
$iF$. The holomorphic structure on $E$ pulls back to an
$\om$-compatible integrable complex structure $J$ on $\UU$, which
satisfies condition (K) by Lemma~\ref{lem:K3}.
\end{proof}

\begin{remark}\label{rem:K1}
Proposition~\ref{prop:K} should also hold if $L$ is 1-sheeted near $K$,
but the proof will be more involved in that case.
\end{remark}

\begin{remark}\label{rem:K2}
Consider a submanifold $K \subset Q$ and the immersed Lagrangian
$L=Q\cup NK\subset T^*Q$ as in Example~\ref{ex:cot}. Suppose that $K$
admits a flat metric and has trivial normal bundle (e.g.~for a 1-knot
in $\R^3$). Pick a flat metric on a neighbourhood of $K$ in $Q$ for
which $K$ is totally geodesic and let $J$ be the (integrable!) complex
structure on a neighbourhood of $K$ in $T^*Q$ induced by this
metric. Then local isometric coordinates for $Q$ mapping $K$ to $\R^k$
extend to local holomorphic coordinates satisfying condition (K).
\end{remark}

For the remainder of this section, we consider $(X,L,\om,J)$
satisfying conditions (X), (Y), (L), (K) and (R) above.

\subsection{Area and energy}
%%%%%%%%%%%%%%%%%%%%%%%%%%%%%%%%%%%%%%%%%%%%%%%%%%%%%%%%%%%%%%%%%%%%%

%Recall that we consider symplectic manifolds $X=\bar X \cup (\R_+ \x
%M)$ and cleanly immersed Lagrangian submanifolds $L \subset X$ with $L
%\cap (\R_+ \x M) = \R \x \Lambda$ for some embedded compact
%submanifold $\Lambda \subset M$ with $\lambda|_\Lambda=\om|_\Lambda=0$.

Recall from \cite{BEHWZ} that the {\em (Hofer) energy} of a
holomorphic curve $f$ is defined as a sum of two terms,
$$
E(f) := E_\om(f) + E_\lambda(f).
$$
When $f=(f_\R,f_M):(S,\p S,j) \to (\R \x M, \R \x \Lambda, J)$, we set
$$
E_\om(f):= \int_Sf_M^*\om, \quad E_\lambda(f):= \sup_{\pHi \in \CC}
\int_S (\pHi \circ f_\R)df_\R \wedge f_M^*\lambda,
$$
where the supremum is taken over the set $\CC$ of nonnegative functions
$\pHi:\R \to \R$ with
$$
\int_\R \pHi(s)ds =1.
$$
Similarly, for a holomorphic curve $f:(S,\p S,j) \to (X, L, J)$
we define its {\em $\om$-energy (or area)}
$$
E_\om(f):= \int_{f^{-1}(\bar X)}f^*\om + \int_{f^{-1}(\R_+ \x M)}f_M^*\om
$$
and its {\em $\lambda$-energy}
$$
E_\lambda(f):=\sup_{\pHi \in \CC^+}\int_{f^{-1}(\R_+ \x M)} (\pHi
\circ f_\R)df_\R \wedge f_M^*\lambda,
$$
where the supremum is taken over the set $\CC^+$ of all nonnegative
functions $\pHi:\R_+ \to \R$ with
$$
\int_{\R^+}\pHi(s)ds = 1.
$$
Since the almost complex structure is comptible with $\omega$ and since $J$ pairs the symplectization- and the Reeb direction in the ends of $X$, it follows that $E_\om(f)\ge 0$ and $E_\lambda(f) \geq 0$ for any holomorhpic $f$. Moreover, $E_\om(f)=0$ implies that either $f$
is constant, or the image of $f$ is contained in some
cylinder over a closed Reeb orbit or in some strip over a Reeb chord.

\subsection{Monotonicity and removal of singularities}
%%%%%%%%%%%%%%%%%%%%%%%%%%%%%%%%%%%%%%%%%%%%%%%%%%%%%%%%%%%%%%%%%%%%%

\begin{lemma}[Monotonicity Lemma]\label{lem:mon}
There exist constants $\eps_M,C_M>0$ depending only on $(X,L,\om,J)$
with the following property: For any
$J$-holomorphic map $f:(S,\p S)\to(X,L)$ from a (possibly
noncompact) Riemann surface with boundary, passing through a point
$x\in\bar Y$ and such that $f^{-1}(B(x,r))$ is compact for some
$r<\eps_M$, we have
$$
   E_\om(f)\geq C_Mr^2.
$$
\end{lemma}

\begin{proof}
The proof in ~Proposition 4.7.2 in~\cite{Sik} carries over to the
present setting as follows. Since the metric is smooth, there exist
constants $C_0$ and $C_1$ and $r_0>0$ with the following properties
for $0<r<r_0$ and any $x\in \bar Y$:  $B(x,r)\cap L$ is contained in
a contractible subset of $B(x,2r) \cap L$, for every pair of points
$x,y\in B(x,r) \cap L$ there is a curve in $B(x,2r) \cap L$ of length at most
$C_1 d(x,y)$ connecting them, and every closed curve $\gamma$ in
$B(x,2r)$ bounds a disk in $B(x,2r)$ of area at most
$C_0\ell^2(\gamma)$, where $\ell(\gamma)$ denotes the length of $\gamma$.

Assume that $x\in f(S)$. For $r<r_0$, let $S_r=f(S)\cap
B(x,r)$, $\alpha_r=\p B(x,r)\cap f(S)$, and
$\beta_r=f(\p S)\cap B(x,r)$. If $r$ is choosen generically then
$\alpha_r$ and $\beta_r$ are collections of smooth curves.
For each component $\alpha'_r$ of $\alpha_r$ we choose a
curve $\gamma'_r$ in $L\cap B(x,2r)$ of length at most $C_1
l(\alpha'_r)$, where $l(\alpha'_r)$ is the length of $\alpha'_r$. Then
$\alpha'_r\cup\gamma'_r$ is a closed curve in $B(x,2r)$ of length at
most $(1+C_1)l(\alpha'_r)$. Let $\gamma_r$ denote the union of all
curves $\gamma'_r$. Then by assumption $\alpha_r\cup\gamma_r$ bounds a
collection $D$ of disks in $B(x,r)$ of total area at most
$Cl^{2}(\alpha_r)$. Similarly, $\beta_r \cup \gamma_r$ is a cycle in
the contractible set $B(x,2r) \cap L$, and so it bounds a surface $N$
in $L\cap B(x,2r)$. By Stokes' theorem
\[
\int_{S_r\cup D\cup N}\omega=0.
\]
Clearly $\int_N\omega=0$. Moreover, since $\om$ is a calibration,
$|\int_D\omega|$ is bounded by the area of $D$,
and we conclude that
\[
\int_{S_r}\omega\le C\,\ell^{2}(\alpha_r),
\]
for some constant $C$.
Consider the distance function $\rho$ from $x$. Since the norm of the
gradient of $\rho$ in the ambient manifold is $1$ we conclude that
$|\nabla \rho|\le 1$ on $S_r$. So if we let $a(\rho)$ denote the area
of $S_\rho$ then, by Sard's theorem and the coarea formula, we have
$a'(\rho) \geq \ell(\alpha_\rho)$ for almost every $\rho \leq r$. Consequently
\[
\frac{d\sqrt{a}}{d\rho}=\frac{a'(\rho)}{2\sqrt{a(\rho)}}\ge \frac{1}{2\sqrt{K}},
\]
for some constant $K$. Integrating we find $a(r)\ge C_M r^{2}$.
\end{proof}

\begin{remark}
Alternatively, Lemma~\ref{lem:mon} can be proved by using Proposition
4.7.2 in~\cite{Sik} outside $K$ and condition (K) near $K$. This
reduces the lemma to the case of a holomorphic map
$f=(f_1,\dots,f_n):(S,\p S)\to(\C^n,\R^n\cup \R^k\times
i\R^{n-k})$ passing through the origin, which can be proved by
considering the componentwise square
$g:=(f_1^2,\dots,f_n^2):(S,\p S)\to(\C^n,\R^n)$ with smooth
Lagrangian boundary condition.
\end{remark}

Let $D:=\{z\in\C\mid |z|< 1\}$ and $D^+:=\{z\in D\mid \Im(z)\geq 0\}$.

\begin{lemma}[Removal of singularities]\label{lem:remsing}
(a) Let $f:D\setminus\{0\}\to X$
be continuous and $J$-holomorphic in the interior with finite energy
$E(f)<\infty$.
If $f$ is bounded, then it extends to a continuous map $D\to X$.

(b) Let $f:(D^+\setminus\{0\},D^+\cap\R\setminus\{0\})\to(X,L)$
be continuous and $J$-holomorphic in the interior with finite energy
$E(f)<\infty$.
If $f$ is bounded, then it extends to a continuous map $D^+\to X$.
%
%(b) If $f$ is unbounded, then it is asymptotic near $0$ to the trivial
%strip over a Reeb chord.
\end{lemma}

\begin{proof}
Both cases follow from the argument given right after Theorem 4.1.2
in~\cite{MS}, using the Monotonicity Lemma~\ref{lem:mon} above.
\end{proof}

\subsection{Asymptotics}
%%%%%%%%%%%%%%%%%%%%%%%%%%%%%%%%%%%%%%%%%%%%%%%%%%%%%%%%%%%%%%%%%%%%%
%Set $\R_+:=[0,\infty) \subset \R$.
We have the
following descriptions of the asymptotic behavior of a holomorphic
curve $f\colon (S,\pa S)\to (X,L)$ where $(X,L,\omega,J)$ satisfies conditions (X), (Y), (L), (K) and
(R) above near a non-removable puncture (cf.~\cite[Prop~5.6]{BEHWZ}).
\begin{prop}\label{prop:asymptotics}
{$\quad$}
\begin{itemize}
\item[$({\rm a})$] Let $f:\R_+ \x S^1\to (Y,J)$ be a holomorphic curve with $E(f)<
  \infty$ and suppose the image of $f$ is unbounded. Then $f(s,t) \in
  \R_+ \x M$ for all sufficiently large $s$, and there exists
  $T>0$ and a periodic orbit $\gamma$ of the Reeb vector field of
  period $T$ such that
$$
\lim_{s \to \infty}\pi_M \circ f(s,t) = \gamma(Tt), \quad \lim_{s\to
  \infty}\frac {\pi_\R \circ f(s,t)}{s} = T
$$
in $C^\infty(S^1)$.
\item[$({\rm b})$]
 Let $f:\R_+ \x [0,1],\R_+ \x \{0,1\})\to (Y,L,J)$ be a
  holomorphic curve with $E(f)< \infty$ and suppose the image of $f$
  is unbounded. Then $f(s,t) \in \R_+ \x M$ for all sufficiently large
  $s$, and there exists $T>0$ and a Reeb chord $\gamma$ of
  $\Lambda\subset M$ of length $T$ such that
$$
\lim_{s \to \infty} \pi_M \circ f(s,t) = \gamma(Tt), \quad \lim_{s\to
  \infty}\frac {\pi_\R \circ f(s,t)}{s} = T
$$
in $C^\infty([0,1])$.
\end{itemize}
\end{prop}

\begin{proof}
Case (a) is proved in~\cite{HWZ-asymptotics}. Case (b) is proved
in~\cite{Abbas} in case $\dim(Y)=3$ and the proof there carries over
to the higher dimensional situation with only minor changes. 
% In case (b) a standard argument (see~\cite{Hof}) yields a sequence
% $s_k\to \infty$ such that $f$ maps the segments $\{s_k\} \x [0,1]$ to
% $\R_+ \x M$, with the $\R$-components tending to infinity and the
% $M$-components converging to some Reeb chord
% $\gamma$. The Monotonicity Lemma~\ref{lem:mon} shows that if the
% rectangle between two such segments leaves $\R_+\times M$, then its area is
% uniformly bounded from below. So this can only happen finitely many
% times. It follows that there exists $S>0$ such that $[S,\infty) \x
% [0,1]$ is mapped entirely into $\R_+\times M$ (in particular, it does
% not meet $K$) and the usual proof in~\cite{HWZ-asymptotics} gives the
% convergence.
\end{proof}

\subsection{Quantization of energy}
%%%%%%%%%%%%%%%%%%%%%%%%%%%%%%%%%%%%%%%%%%%%%%%%%%%%%%%%%%%%%%%%%%%%%

\begin{lemma}\label{lem:quant}
There exists a constant $\hbar>0$, depending only on $(X,L,\om,J)$,
such that for every proper $J$-holomorphic
map $f: (S,\p S) \to (X,L)$
$$
E_\om(f) \geq \hbar.
$$
\end{lemma}
\begin{proof}
The proof of Lemma 4.2 in~\cite{CM-comp} for symplectizations directly
carries over to the relative case, giving the result for curves whose
image is contained in the end $\R_+ \x M$. For curves whose image meets
$\bar X$, the lower energy bound is guaranteed by the Monotonicity
Lemma~\ref{lem:mon}.
\end{proof}

\begin{lemma}\label{lem:quant2}
For every $E>0$ there exists a constant $\hbar(E)>0$, depending only on
$(X,L,\om,J)$ and $E$, such that the area of every proper
$J$-holomorphic cylinder or strip $f: (S,\p S) \to (X,L)$ with
$E(f)\leq E$ and $E_\om(f)>0$ satisfies
$$
E_\om(f) \geq \hbar(E).
$$
\end{lemma}

\begin{proof}
The proof of Lemma 10.9 in~\cite{BEHWZ} resp.~Lemma 4.6 in~\cite{CM-comp}
carries over to the relative case.
\end{proof}

\subsection{Holomorphic cylinders and strips of small area}
%%%%%%%%%%%%%%%%%%%%%%%%%%%%%%%%%%%%%%%%%%%%%%%%%%%%%%%%%%%%%%%%%%%%%

Finally, we need the following generalization of a result of Hofer,
Wysocki and Zehnder.

\begin{prop}\label{prop:smallarea}
Given $E_0,\eps>0$ there are constants $\sigma,c>0$ with the following
properties:

(a) For every $R>c$ and every holomorphic cylinder $f:[-R,R] \x S^1
\to \R \x M$ satisfying $E_\om(f)\leq \sigma$ and $E(f) \leq E_0$
there exists either a periodic Reeb orbit $\gamma$ of period $T>0$ such that
$\pi_M \circ f(s,t) \in B_\eps(\gamma(Tt))$ or some point
$p\in M$ such that $\pi_M \circ f(s,t) \in B_\eps(p)$ for all $s\in
[-R+c,R-c]$ and all $t\in S^1$.

(b) For every $R>c$ and every holomorphic strip $f:([-R,R] \x
[0,1],[-R,R] \x \{0,1\} \to (\R \x M,\R \x \Lambda)$ satisfying the
inequalities $E_\om(f) \leq \sigma$ and $E(f) \leq E_0$ there exists
either a Reeb chord $\gamma$ of length $T> 0$ such that $\pi_M \circ
f(s,t) \in B_\eps(\gamma(Tt))$ or some point $p \in \Lambda$ such that
$\pi_M \circ f(s,t) \in B_\eps(p)$ for all $s\in [-R+c,R-c]$ and all
$t\in [0,1]$.
\end{prop}
\begin{proof}
The proof of (a) in~\cite{HWZ-smallarea} (for the contact case)
and~\cite{BEHWZ} (for general stable Hamiltonian structures) carries
over to case (b).
\end{proof}

We will also use the following version of Lemma~5.14. of \cite{BEHWZ},
whose proof carries over to the relative case using the Monotonicity
Lemma~\ref{lem:mon}.
\begin{lemma}\label{lem:longstrips}
Let $u_n:([-n,n] \x [0,1],[-n,n] \x \{1,0\}) \to (X,L)$ be a sequence
of $J$-holomorphic strips with
\begin{enumerate}[(i)]
\item $\lim_{n \to \infty} E_\om(u_n) = 0$, and
%\item the $u_n$ satisfy uniform bounds of the form \eqref{eq:thin1}, and
\item $\lim_{n \to \infty} u_n|_{\{\pm n\} \x [0,1]} = p_\pm \in L$ in
  $C^{\infty}([0,1],X)$.
\end{enumerate}
Then $\lim_{n\to \infty}\diam u_n([-n,n] \x [0,1])= 0$, and in
particular $p_+=p_-$.\qed
\end{lemma}
%

%%%%%%%%%%%%%%%%%%%%%%%%%%%%%%%%%%%%%%%%%%%%%%%%%%%%%%%%%%%%%%%%%%%%%
\section{Compactness}\label{sec:compact}
%%%%%%%%%%%%%%%%%%%%%%%%%%%%%%%%%%%%%%%%%%%%%%%%%%%%%%%%%%%%%%%%%%%%%

In this section, we apply the above local theory to establish
a compactness result for holomorphic curves. We consider $(X,L,\om,J)$
satisfying conditions (X), (Y), (L), (K) and (R) from the previous
section.
Without loss of generality we assume that the neighbourhood in
condition (K) is $\UU=\UU_\eps$, where $\UU_r$ denotes the open
$r$-neighbourhood of $K$ in $X$ with respect to the distance induced
by $(\om,J)$.

\subsection{$C^\infty_\loc$-convergence}
%%%%%%%%%%%%%%%%%%%%%%%%%%%%%%%%%%%%%%%%%%%%%%%%%%%%%%%%%%%%%%%%%%%%%

Consider a fixed connected Riemann surface $(\Sigma,j)$ of finite
type, by which we mean the complement of a finite set of points (the
``punctures'') in a connected compact Riemann surface with
boundary. We assume that $\Sigma$ is stable, meaning that its double
is a stable punctured Riemann surface in the usual sense. It follows
that $\Sigma$ admits a unique complete hyperbolic metric $h_j$
compatible with $j$ such that each component of $\p \Sigma$ is a
geodesic (either closed or infinite).

Suppose $f_n:(\Sigma,\p \Sigma, j), \to (Y,L,J)$ is a sequence of
continuous maps which are $J$-holomorphic on $\inn \Sigma$ and have
finite energy. Recall (Proposition~\ref{prop:asymptotics}) that
if an interior puncture is a non-removable
singularity of the map $f_n$, then $f_n$ will be asymptotic to a
trivial cylinder over a closed Reeb orbit in a neighborhood of that
puncture. Similarly, near non-removable boundary punctures $f_n$ is
asymptotic to a trivial strip over a Reeb chord. We make
the following additional assumptions on our sequence:
\begin{enumerate}[(S1)]
\item Each puncture is either removable for all $n \geq 1$ or
  non-removable for all $n \geq 1$, and at non-removable punctures the
  asymptotic Reeb chords resp.~closed Reeb orbits are independent of
  $n$.
\item There exists a constant $C>1$ such that for all
  $z\in \inn \Sigma$ and all $n\geq 1$ we have
\begin{gather}
| \nabla f_n(z)| \leq \frac{C}{\rho(z)} \qquad \text{\rm if
}f_n(z)\notin\UU_{\eps/4}, \label{eq:bounds1}\\
|\nabla (\tau\circ f_n)(z)| \leq \frac{C}{\rho(z)} \qquad
\text{\rm if }f_n(z)\in\UU_\eps.  \label{eq:bounds2}
\end{gather}
Here $\rho(z)$ denotes the injectivity radius at $z\in \Sigma$ in the
hyperbolic metric, the norm is computed with respect to the
hyperbolic metric on the domain and the metric determined by $\om$ and
$J$ on the target, and $\tau:\UU_\eps\to T^*K$ is the holomorphic
projection appearing in condition (K).
%\item There exists some compact subset $B\subset Q$ such that for the
%  natural projection $p:T^*Q \to Q$ the images $p \circ
%  f_n(\sigma)$ are contained in $B$.
\end{enumerate}

\begin{prop}\label{prop:convergence}
For any sequence of maps $f_n:(\Sigma,\p \Sigma,j) \to (Y,L,J)$
satisfying (S1)-(S2) there exists a subsequence, still denoted $f_n$, and
a map $f:(\Sigma,\p \Sigma,j) \to (Y,L,J)$ such that
\begin{itemize}
\item  $f$ is continuous on $\Sigma$ and holomorphic on $\inn
\Sigma \cup (\p \Sigma \setminus f^{-1}(K))$,
\item $f_n \to f$ in $C^0_\loc$ on $\Sigma$, and
\item $f_n \to f$ in $C^\infty_\loc$ on $\inn \Sigma \cup (\p \Sigma
  \setminus f^{-1}(K))$.
\end{itemize}
\end{prop}
\begin{proof}
We fix an exhaustion $B_1 \subset B_2 \subset B_3 \subset \dots $ of
$\Sigma$ by closed subsets $B_j:=\{z\in \Sigma\colon d(z_0,z)\le j\}$,
where $d(z_0,z)$ denotes the distance between some fixed point $z_0$
and $z$. Since $B_j$ is compact, the injectivity radius 
$\rho$ is bounded below on it by 
$$
   \rho_j := \min_{B_j}\rho > 0.
$$
So with $C_j:=C/\rho_j$ condition (S2) yields the following gradient
bounds for $z \in B_j \cap \inn \Sigma$ and all $n$:
\begin{gather}
|\nabla f_n(z)| \leq C_j \qquad \text{\rm if
}f_n(z)\notin\UU_{\eps/4}, \\
| \nabla (\tau\circ f_n)(z)| \leq C_j \qquad
\text{\rm if }f_n(z)\in\UU_\eps.  \label{eq:bounds4}
\end{gather}
We now distinguish two cases.

Case 1: The sequence $f_n(z_0)$ is unbounded. \\
Then, after passing to a subsequence, we have $f_n(z_0)\in\R_+\times M$
with $\R$-component going to infinity. By the gradient bounds on
$B_j$, for each fixed $j$ we have $f_n(B_j)\subset\R_+\times M$ for
all sufficiently large $n$ with
$\R$-component going uniformly to infinity. Hence we
can apply the usual compactness argument with smooth Lagrangian boundary
conditions $\R \x \Lambda$ in the symplectization $\R \x M$.

Case 2: The sequence $f_n(z_0)$ remains in a compact subset $A\subset
X$. \\
By the gradient bounds on $B_j$, for each fixed $j$ the images
$f_n(B_j)$ remain in the compact subset
$$
   A_j := \{x\in X\mid d(x,A\cup\bar\UU_\eps)\leq C_j\,j\}.
$$
%After
%passing to a further subsequence we may assume that $f_n(z_0)$
%converges to some $x \in T^*Q$.
%A given ball $B_j$ in our exhaustion will be mapped to the ball of
%radius $Cj$ around $x\in T^*Q$ by each $f_n$. In other words, $f_n|_{B_j}:B_j \to
%B(x,Cj) \subset T^*Q$ is a sequence of $C$-Lipschitz maps between
%compact metric spaces, and so some subsequence has a $C^0$ limit
%$f:(B_j,B_j \cap \p \Sigma) \to (T^*Q,L)$.
%
%To prove holomorphicity of the limit, we argue as follows.
%
%Denote by $\UU_\eps$ the neighborhood of $K$ which is holomorphically
%parametrized according to assumption (A), and for each
%$0<r<\eps$ let $\UU_r$ be the corresponding $r$-neighborhood of $K$.
%By making $\eps>0$ smaller if necessary, we can assume in addition
%that every point $x\in K\subset \UU_\eps$ has a neighborhood
%biholomorphic to $B^{2d}(0,\eps) \x D^2_\eps \x \dots \x D^2_\eps$,
%where $d=\dim K$ and the first ball corresponds to the components in
%$\nu_\eps(K)$.
For each $z\in B_j$, we define the open ball
$$
   S_z := \inn B(z,\frac{\eps}{4C_j}).
$$
Now $B_j$ is covered by a finite collection $S_{z_1},\dots,S_{z_r}$
of these sets.
%, and we can arrange in addition that $f(x_i)\in K$ if
%$f(U_{x_i}) \cap K \neq \varnothing$.

For each of the points $z_i$, exactly one of the following two things
happens:
\begin{enumerate}[(a)]
\item after passing to a subsequence $n_k$, $f_{n_k}(z_i) \notin
  \UU_{\eps/2}$ for all $k \geq 1$, or
\item there exists some $N(z_i)$ such that $f_n(z_i)\in \UU_{\eps/2}$
  for all $n \geq N(z_i)$.
\end{enumerate}
If $z_1$ is of type (a), then we pass to the subsequence $f_{n_k}$,
and if $z_1$ is of type (b) we pass to the subsequence $f_n$ with $n
\geq N(z_1)$. Repeating this for each index $i=2,\dots,r$, we arrive
at the situation where for each $z_i$ either (a) or (b) holds for all
$n \geq 1$.

Consider first $z_i$ of type (a). Then the gradient bounds imply that
$f_n(S_{z_i}) \cap \UU_{\eps/4} = \varnothing$ for all $n\geq 1$.
So the maps $f_n:S_{z_i}\to X$ have smooth Lagrangian boundary
conditions, and the usual compactness argument yields a subsequence
which converges in $C^\infty_\loc$ (up to the boundary) to a
holomorphic limit map.

Next we consider $z_i$ of type (b). We claim that
$f_n(S_{z_i}) \subset \UU_\eps$ for all $n \geq 1$. To see this,
consider $z\in S_{z_i}$ and a constant speed minimal geodesic
$\gamma:[0,1]\to \Sigma$ from $z_i$ to $z$, set
$$
t':= \sup \{t \in [0,1] \,:\, d(f_n(\gamma(t)),K)\leq \frac \eps 2 \},
$$
and compute
\begin{align*}
d(f_n(z),K)
&\leq \frac \eps 2 + d(f_n(\gamma(t')),f_n(y)) \\
&\leq \frac \eps 2 + \int_{t'}^1 |\nabla
f_n(\gamma(t))|\,|\dot\gamma(t)| dt\\
&\leq \frac \eps 2 + d(z_i,z) \cdot
\max_{t\in [t',1]} |\nabla f_n(\gamma(t))| \\
&\leq \frac \eps 2 + \frac {\eps}{4C_j} \cdot C_j\\
&\leq \eps.
\end{align*}
This proves the claim. Now consider the holomorphic maps
$$
   \tau\circ f_n:S_{z_i}\to T^*K,
$$
where $\tau:\UU_\eps\to T^*K$ is the holomorphic projection in
condition (K). These maps have smooth Lagrangian boundary conditions
on $K$, are uniformly boun\-ded, and have uniform gradient
bounds by condition \eqref{eq:bounds4} above. So the usual compactness argument
yields a convergent subsequence. Denote the limit map by $g:S_{z_i}\to
T^*K$.

It remains to show convergence of the components transverse to
$T^*K$. For this, note that by condition (K) each $x\in T^*K$ has a
neighbourhood $U_x$ with a holomorphic embedding $\tau^{-1}(U_x)\into
U_x\times\C^{n-k}$ sending the branches of $L$ to $K\times\R^{n-k}$
and $K\times i\R^{n-k}$, where $k=\dim K$. Cover the image of $\tau$
by finitely many such neighbourhoods $U_{x_1},\dots,U_{x_s}$ and denote by
$\nu_\ell:U_x\times\C^{n-k}\to\C$ the holomorphic projection onto the
$\ell$-th $\C$-factor, $\ell=1,\dots,n-k$.
Pulling back the $U_{x_m}$ under $g$
yields open subsets $S_{z_i,m}\subset S_{z_i}$ and holomorphic
functions $\nu_\ell\circ f_n:S_{z_i,m}\to\C$ mapping the boundary
$S_{z_i,m}\cap \p\Sigma$ to $\R\cup i\R$. Since the functions are also
uniformly bounded, Lemma~\ref{lem:3} yields a convergent subsequence
for each $\ell=1,\dots,n-k$ and $m=1,\dots,s$.

Combining types (a) and (b), we conclude that $f_n:B_j\to X$ has a
subsequence converging in the desired sense to a continuous limit map
$f_j:(B_j,B_j\cap \p \Sigma) \to (X,L)$ which is
holomorphic on $B_j \cap (\inn \Sigma \cup (\p \Sigma \cap
f_j^{-1}(K)))$.
Finally, we take a diagonal sequence with respect to the index $j$ of
$B_j$ in our exhaustion to get a subsequence converging on all of
$\Sigma$ to a limit map $f:(\Sigma,\p\Sigma)\to (X,L)$, with
convergence in $C^0_\loc$ on $\Sigma$ and in
$C^\infty_\loc$ on $\inn \Sigma \cup (\p \Sigma \cap f^{-1}(K))$.
\end{proof}

\begin{remark}
Note that the same proof works if we allow the domains of $f_n$ to
vary in a converging sequence of Riemann surfaces $(\Sigma_n,j_n)$.
\end{remark}

\subsection{Proof of the Compactness Theorem~\ref{thm:comp}}
%%%%%%%%%%%%%%%%%%%%%%%%%%%%%%%%%%%%%%%%%%%%%%%%%%%%%%%%%%%%%%%%%%%%%

For the remainder of this section, we assume familiarity with the
proof of compactness for holomorphic curves in SFT presented in
\cite{BEHWZ}, and we will sketch how it can be adapted to our setting.
We freely use the concepts and notation of \cite{BEHWZ}.

We denote a {\em nodal Riemann surface} by $(S,j,D,M)$, where $(S,j)$
is a compact Riemann surface, $D$ is the set of double points, and $M$
is the set of marked points.
As our curves have boundary, a nodal Riemann surface will
have nodes of two types: boundary nodes, where both points are on the
boundary, and interior nodal points, where both are in the
interior. We do not consider mixed nodes. Also, we think of the
boundary components as ordered, and so the set of marked points $M$
can be split as $M=M_\inn \cup M_1 \cup \dots \cup M_b$, where $b \geq
0$ is the number of boundary components of the surface $S$, and where
the marked points in $M_\inn$ are interior and the marked points in
$M_i$ lie on the $i^{th}$ boundary component. The genus of a nodal
Riemann surface with boundary is the arithmetic genus of the
topological surface obtained by filling each boundary component by a
disc.

We define the {\em signature} of a nodal Riemann surface as the
sequence $\sigma= (g,b;n,m_1,\dots,m_b)$, where $g$ is its genus, $b$
is the number of boundary components, $n$ is the number of interior
marked points, and $m_i$ is the number of marked points on the $i^{th}$
boundary component.

The $\eps$-thin part of every component of a stable nodal Riemann surface
(with respect to its uniformizing metric) now consists of four types
of domains: annuli of finite modulus around a short interior geodesic,
annuli conformally equivalent to the punctured unit disc around each
interior puncture, a rectangular region conformally equivalent to
$[-1,1] \x (-L,L)$ around each short geodesic (minimal in its free
homotopy class) connecting two boundary components, and a region
conformally equivalent to a punctured half-disc $D^+ \setminus \{0\}$
near each boundary puncture.

A {\em decoration} (i.e. an orientation reversing orthogonal
identification of the tangent planes at the two corresponding points)
is required only at interior nodes, since at the boundary the choice of
identification is fixed by matching the boundary directions. we denote
a decorated nodal Riemann surface by $(S,j,D,M,r)$, where $r$ stands
for the decoration.
%we can then
%form the quotient $\hat{S}_D:=S/\{\underline{d_i} \sim
%\overline{d_i}\, |\, i=1,\dots,k\}$ and the deformation $S^{D,r} \to
%\hat{S}_D$ as described in \cite[\S 4.4]{BEHWZ}. Note that $S^{D,r}$ has
%special circles corresponding to interior nodes and special segments
%corresponding to boundary nodes.

We denote by $\overline{\MM}^\$_\sigma$ the moduli space
of connected decorated stable nodal Riemann surfaces with signature
$\sigma$, equipped with the usual topology (\cite{BEHWZ}, cf.~also
\cite{Liu} for the non-decorated case). It is shown in~\cite{BEHWZ}
that, for each fixed signature $\sigma$, the space
$\overline{\MM}^\$_\sigma$ is a compact metric space which coincides
with the closure of its subset $\MM_\sigma$ of smooth marked Riemann
surfaces with boundary of signature $\sigma$. In other words, every
sequence of smooth stable marked Riemann surfaces
$(S_n,j_n,M_n)$ of signature $\sigma$ has a subsequence which
converges to a decorated nodal Riemann surface $(S,j,M,D,r)$ of the
same signature.

%Now consider the symplectic manifold $W=T^*Q$ with the immersed
%Lagrangian submanifold $L=Q \cup N$, where $N \subset T^*Q$ is the
%conormal bundle of some closed submanifold $K \subset Q$. We also fix
%an admissible almost complex structure $J$.
Now consider $(X,L,\om,J)$ as above.
With the above setup for the domains, the definition of a nodal holomorphic
curve of height $1$ in $(X,L,J)$ is exactly the same as in
\cite[\S 8]{BEHWZ}, except that we allow the domain to have boundary, which
is required to be mapped to $L$. Similarly, we get the notion of a
holomorphic building of height $(1|k_+)$, and the notion of
convergence. Fixing the signature $\sigma:=(g,b;n,m_1,\dots,m_b)$, we
obtain the moduli space $\overline{\MM}_\sigma(W,L,J)$ of stable
holomorphic curves of that signature.

Now we can prove the Compactness Theorem~\ref{thm:comp} in the
introduction, which we restate as follows.

\begin{thm}\label{thm:compactness}
Let $(X,L,\om,J)$ satisfy conditions (X), (Y), (L), (K) and (R).
Then for any $E>0$ and for any fixed signature
$\sigma=(g,b;n,m_1,\dots,m_b)$, the space $\overline{\MM}_\sigma(X,L,J) \cap
\{E(f) \leq E\}$ is compact.
\end{thm}

\begin{proof}
The proof closely follows the strategy of the corresponding proof of
Theorem~10.2. of \cite{BEHWZ}. Clearly, it is sufficient to establish
sequential compactness for smooth curves (i.e. without nodes).

So let $f_n:(S_n,\p S_n,j_n) \to (X,L,J)$ be a sequence of curves
of fixed signature and uniformly bounded energy.

{\bf Step 1:} After adding additional marked points if needed, we
may assume that the underlying domains $(S_n,j_n,M_n\cup Z_n)$ of the
$f_n$ are stable.
%A subsequence of these domains will then converge to
%a stable domain $(S,j,M \cup Z,D)$.

Now we want to argue that, by adding a finite set (with
number depending on the energy bound) of additional pairs of points,
one obtains a new sequence of stable domains, denoted by $(S_n,j_n,M_n
\cup Z_n)$, such that the new sequence satisfies the gradient
bounds \eqref{eq:bounds1} and \eqref{eq:bounds2}.

This is based on a bubbling analysis.
Indeed, to achieve \eqref{eq:bounds1} one argues as in \cite{BEHWZ},
producing finite energy planes or spheres that each take a minimal
amount $\hbar>0$ of energy by Lemma~\ref{lem:quant}.

So assume that \eqref{eq:bounds1} holds but \eqref{eq:bounds2} fails,
i.e. there exists a sequence of points $z_n \in S_n$ such that
$f_n(z_n) \in \UU_{\eps/4}$ and $\|\nabla (\tau \circ f_n)(z_n)\| \cdot
\rho(z_n) \to \infty$. After passing to a subsequence, we have one of
the following two cases:
\begin{enumerate}[(i)]
\item $\rho'_n := \frac {\rho(z_n)}{d(z_n,\p S_n)} \leq C < \infty$,
  or
\item $\rho'_n \to \infty$.
\end{enumerate}
In case (i), we find holomorphic embeddings $\phi_n:(D,0) \to
\bigl(S_n\setminus (M_n \cup Z_n),z_n\bigr)$ of the unit disk with
$$
\frac 1 {C'} \rho'_n \leq |\nabla \phi_n| \leq C' \rho'_n
$$
for some constant $C'$, and in case (ii) we find points $\xi_n \in
D^+$ with $\xi_n \to 0$ and holomorphic embeddings
$\phi_n:(D^+,D^+\cap\R,\xi_n) \to \bigl(S_n\setminus (M_n \cup
Z_n),L,z_n\bigr)$ of the upper half disk satisfying the same bounds.

In both cases we can modify the sequence $(z_n)$, rescale
$f_n\circ\phi_n$ as in \cite[\S 10.2.1]{BEHWZ} and apply
Proposition~\ref{prop:convergence} to obtain a $J$-holomorphic plane
$f:\C\to X$ or half-plane $(\H,\R)\to (X,L)$ of finite energy. The map
$f$ is either proper, or it extends to a holomorphic sphere or disk by
Lemma~\ref{lem:remsing}. In either case, $f$ has area
$E_\om(f)\geq \hbar>0$ by Lemma~\ref{lem:quant}. Hence adding a pair
of marked points and repeating this process, we obtain a bound of the
form \eqref{eq:bounds2} after finitely many steps.

{\bf Step 2:}
% As in \cite[\S 10.2.2.]{BEHWZ} it follows from Theorem~\ref{thm:DM}
% that,
After passing to a subsequence, the domains
$(S_n,M_n,Z_n)$ will converge to a decorated nodal
Riemann surface with boundary $(S,j,M,Z,D,r)$. In Step 1 we have
arranged for assumption (S2) to hold for our sequence, and using the
energy bound we can arrange (S1) after passing to a subsequence. So,
by Proposition~\ref{prop:convergence}, for a further subsequence we
obtain $C^\infty_\loc$-convergence on each component of the complement
of the pinching geodesics in $(S_n,j_n,M_n,Z_n)$ of the
maps $f_n$ to some limiting map $f$ defined on the corresponding
components of $(S,j,M,Z,D,r)$.

{\bf Step 3:} Now we have to analyse the convergence in the thin
part. Here, as in \cite{BEHWZ}, one considers each type of
component of the thin part seperately. Annuli near interior marked
points and near interior nodes are treated in detail in \cite[\S
10.2.3]{BEHWZ}. In the other two cases one proceeds analogously, with
the following adaptions.

{\bf Behavior near a boundary node.}
As in the case of interior nodes described in \cite{BEHWZ}, boundary
nodes appear as a result of degeneration of some component of the thin
part of the $S_n$. The associated holomorphic strips $u_n = f_n \circ
\phi_n$, obtained by precomposing with suitable uniformizations
$\phi_n$ whose domains are longer and longer strips, have gradient
bounds of the form
\begin{equation}\label{eq:thin1}
\begin{array}{rclc}
|\nabla u_n(z)| &\leq &C& \text{\rm if }u_n(z)\notin\UU_{\eps/4}, \\
|\nabla (\tau\circ u_n)(z)| &\leq& C & \text{\rm if
}u_n(z)\in\UU_\eps.
\end{array}
\end{equation}
This follows by the same argument as that for equation (35)
in~\cite{BEHWZ}. After passing to a subsequence, the areas
$E_\om(f_n)$ converge to either zero or some positive constant.

First consider the case of zero limiting $E_\om$-energy. If one of the
asymptotics for the limit map $f$ in the adjacent thick parts of $S$
is a Reeb chord, one uses part (b) of Proposition~\ref{prop:smallarea} to
conclude that the other asymptotic equals the same Reeb chord.
If both adjacent asymptotics are points $p_\pm\in L$ one
uses Lemma~\ref{lem:longstrips} to conclude that $p_+=p_-$.

If the limiting $E_\om$-energy of the strips is
positive, in view of the gradient bounds \eqref{eq:thin1} there can be no
bubbling, and so the only possibility is breaking into a sequence
of holomorphic strips. By Lemma~\ref{lem:quant2}, each nontrivial
strip carries area at least $\hbar(E)$, so there can only be finitely
many of them.

{\bf Behavior near a boundary puncture.} Here, the adjustments are
similar in nature to the ones described for the previous case, and we
omit the details.

After Step 3 is done, we have a subsequence $f_n$ of the original
sequence of holomorphic curves converging to a limiting map $f$
defined on some nodal Riemann surface $(S,j,M,Z,D,r)$ such that $\lim
E(f_n)=E(f)$.

{\bf Step 4:} It remains to recover the level structure in the
holomorphic building $f$ constructed above, and this is done exactly
as in \cite[\S~10.2.5]{BEHWZ}.
\end{proof}

%%%%%%%%%%%%%%%%%%%%%%%%%%%%%%%%%%%%%%%%%%%%%%%%%%%%%%%%%%%%%%%%%%%%%
\section{Proof of the Finiteness Theorem~\ref{thm:fin}}\label{sec:finite}
%%%%%%%%%%%%%%%%%%%%%%%%%%%%%%%%%%%%%%%%%%%%%%%%%%%%%%%%%%%%%%%%%%%%%

As before, we consider $(X,L,\om,J)$ satisfying conditions (X), (Y),
(L), (K) and (R). Now we assume in addition
that $(X,L,\om=d\lambda)$ is {\em exact with convex end},
i.e.~$\lambda$ is a positive contact form on $M$ which extends as a
primitive of $\om$ to $\bar X$.

For a holomorphic curve $f:(S,\p S,j)\to(X,L,J)$ a {\em switch} is a
point in $\p S$ which is mapped to $K$.

Now we can prove the Finiteness Theorem~\ref{thm:fin} in the
introduction, which we restate for convenience.

\begin{thm}\label{thm:fin2}
In the situation of Theorem~\ref{thm:comp}, suppose in addition
that\linebreak
$(X,L,\om=d\lambda)$ is exact with convex end. Then for each
$s\in \N$ and $C>0$ there exists a constant $\kappa(s,C)\in\N$ such
that every holomorphic {\em disk} $f:(\dot D,\p\dot D,j)\to (X,L,J)$
with at most $s$ boundary punctures and energy $\leq C$ has at most
$\kappa(s,C)$ switches.
\end{thm}

\begin{proof}
We argue by contradiction. So assume there exists a sequence of
holomorphic disks $f:(\dot D,\p\dot D,j)\to (X,L,J)$
with at most $s$ boundary punctures and energy $\leq C$
such that $f_n^{-1}(K) \cap \p D$ contains at least $n$ points.
After passing to a subsequence, we may assume that the number $s$ of
boundary punctures and the ordered collection of asymptotic Reeb
chords $\Gamma=(\gamma_1,\dots,\gamma_s)$ is fixed in the sequence.

By Theorem~\ref{thm:compactness} in the previous section, some subsequence
of the $f_n$ converges to a stable holomorphic curve $f$ of some
finite height $(1|k)$, whose domain is a disc-like nodal Riemann surface
$(S,j,Z,D,r)$ with $Z \subset \p S$ of cardinality $s$. The convergence
is in $C^0$ and in $C^\infty_\loc$ away from the punctures, the nodes
and $f^{-1}(K) \cap \p S$.

Consider a component $C$ of $S$ on which $f$ is non-constant. We claim
that in this case $f^{-1}(K)\cap\p C$ is finite. To see this, suppose
otherwise. Since $f$ tends to infinity near the boundary punctures the
set $f^{-1}(K)\cap\p C$ avoids a neighbourhood of the punctures and
thus has a limit point $p\in\p C$. Pick a neighbourhood $S_p$ of $p$
which is mapped into a neighbourhood as in condition (K2) on which
we have holomorphic coordinates mapping the branches of $L$ to $\R^n$
and $\R^k\times i\R^{n-k}$. Consider for $\ell=k+1,\dots,n$ the
holomorphic map $\nu_\ell\circ f:S_p\to\C$, where $\nu_\ell:\C^n \to
\C$ is the projection onto the $\ell$-th $\C$-factor in these
coordinates. Since $\nu_\ell\circ f$ has infinitely many zeroes in $S_p$,
Lemma~\ref{lem:4} implies that it vanishes identically, so
$f(S_p) \subset K^\C$, where $K^\C\subset\UU_\eps$ is the
complexification of $K$ in condition (K).
%Let $S_0$ be the component of the limiting domain $S$ containing
%$U_{x_i}$.
By unique continuation, the component of $p$ in $f^{-1}(\UU_\eps))$ is
mapped into $K^\C$, so in
particular the (connected) boundary of $C$ is mapped entirely into
$K$. But since $L$ is exact, the boundary of a
nonconstant component must contain at least one positive
puncture. This contradiction completes the proof of the claim.

It follows that $f^{-1}(K)\cap\p S$ consists of finitely many points
and finitely many components on which $f$ takes a constant value on
$K$. Pick disjoint compact sets $S_1,\dots,S_r$ with piecewise smooth
boundary such that $f^{-1}(K)\cap\p S\subset\cup_i\inn S_i$ and for each
$S_i$ one of the following holds:

(a) $S_i$ contains precisely one point of $f^{-1}(K)\cap\p S$ and no
nodes, or

(b) $S_i$ contains precisely one connected union of components on
which $f$ takes a constant value on $K$.

Moreover, we may assume that each $S_i$ is mapped into the
neighbourhood $\UU_\eps$ of $K$ in condition (K).
Note that for $n\geq N$ sufficiently large we have $f_n(\p D \setminus
\cup_iS_i) \cap K=\varnothing$ by the $C^0$-convergence on compact sets.
Since $f_n^{-1}(K) \cap \p D$ contains at least $n$ points, it
follows that in some $S_i$ the map $f_n$ has at least $n/r$ points of
$\p D$ mapping to $K$.

Suppose first that this $S_i$ is of type (a). Then $f_n\to f$ in
$C^\infty$ on $S_i$ and (composing as above with projections to $\C$)
Lemma~\ref{lem:4} implies that $f$ maps $S_i$ into $K^\C$. As above,
this yields a contradiction.

Finally, suppose that $S_i$ is of type (b). Then we argue as in the
proof of Lemma~\ref{lem:4}: By Lemma~\ref{lem:wind},
the winding number of $f_n$ over $\Gamma_i:=\p S_i\setminus(S_i\cap\p
S)$ satisfies
$$
   w(f_n,\Gamma_i) \geq \frac{n}{4r} \stackrel{n
   \to\infty}\longrightarrow \infty.
$$
On the other hand, the smooth convergence $f_n\to f$ on $\Gamma_i$
implies
$$
   w(f_n,\Gamma_i) \stackrel{n \to\infty}\longrightarrow
   w(f,\Gamma_i)<\infty.
$$
This contradiction completes the proof of the Finiteness Theorem.
\end{proof}

\begin{corollary}\label{cor:finite}
In the situation described above, for every ordered collection of Reeb chords
$\Gamma=(\gamma_1,\dots,\gamma_s)$, there exists a constant
$\kappa=\kappa(\Gamma)$ with the following property.
If $Z \subset \p D^2$ has cardinality $s$ and $f:(D^2\setminus Z,\p
D^2\setminus Z) \to (T^*\R^3,L)$ is a $J$-holomorphic disc with
asymptotics $\Gamma$,
then $f^{-1}(K) \cap \p D^2$ contains at most $\kappa$ points. In
particular, $f$ has at most $\kappa$ switches.
\end{corollary}

\appendix
\section{Dimensions of moduli spaces}\label{sec:app}
Consider a quadruple $(X,L,\om,J)$ satisfying conditions (X), (Y),
(L), (K) and (R) in Section~\ref{sec:global}.
In this appendix we give the dimension formula for moduli spaces of
holomorphic curves in $X$ with {\em smooth} boundary on $L$, interior
punctures asymptotic to closed Reeb orbits, boundary punctures
asymptotic to Reeb cords, and {\em Lagrangian intersection punctures},
i.e.~boundary punctures asymptotic to the clean self-intersection $K$.
Before proving the formula we introduce the (topological)
data needed to state it.

Consider a Lagrangian intersection puncture mapping to a point $k$ in
a component $K_{d}$ of $K$, where $\dim(K_d)=n-d$. We restate the
relevant part of condition (K) as follows:
\begin{itemize}
\item[$({\rm t0})$]
Near $k$ there exist local holomorphic coordinates
$\C^{n-d}\times\C^{d}$ in which $L$ corresponds to
$\R^{n-d}\times(\R^{d}\cup i\R^{d})$, and $T^{\ast} K$
corresponds to $\C^{n-d}\times\{0\}$.
\end{itemize}
Assume that $f\colon (S,\pa S)\to (X,L)$ is a holomorphic map with a
Lagrangian intersection puncture. Pick a local coordinate $z$ in upper
half plane $\H$ on the source $S$, where the Lagrangian intersection
puncture corresponds to $0\in \H$, and local holomorphic coordinates
$\C^{n-d}\times\C^{d}$ around $f(0)$ in the target as in $({\rm
  t0})$. In these coordinates $f$ is expressed as
\[
f(z)=\bigl(f_1(z),f_2(z)\bigr)\in\C^{n-d}\times\C^{d},
\]
where $f_1$ maps $\R$ to $\R^{n-d}$, and $f_2$ maps $\R_\pm$ to $\R^d$
or $i\R^d$. It follows that $f_1,f_2$ have unique power series
expansions of the form
\begin{equation}\label{eqn:asymptwexp}
   f_1(z) = \sum_{j=0}^\infty a_jz^j,\qquad
   f_2(z) = \sum_{j=0}^\infty c_jz^{j+w},
\end{equation}
%where the sum ranges either over positive integers or over odd
%   positive half integers and
where $a_j\in\R^{n-d}$ for all $j$, either $c_j\in\R^{d}$ for all $j$
or $c_j\in i\R^{d}$ for all $j$, $c_0\neq 0$, and $w$ is either a
positive half integer (if the map $f$ switches local sheets of $L$ at
$k$) or a positive integer (if $f$ remains on one sheet). We call $w$
the {\em asymptotic winding number} of $f$ at the Lagrangian
intersection puncture $k$. (This notion is clearly independent of the
choices involved in its definition).

Assume that the holomorphic map $f\colon (S,\pa S)\to (X,L)$ has
\begin{itemize}
\item $p$ positive interior punctures at Reeb orbits $\gamma_1,\dots,\gamma_p$,
\item $q$ negative interior punctures at Reeb orbits $\beta_1,\dots,\beta_q$,
\item $s$ positive boundary punctures at Reeb chords $c_1,\dots,c_s$,
\item $t$ negative boundary punctures at Reeb chords $b_1,\dots,b_t$, and
\item $l$ Lagrangian intersection punctures on the boundary mapping to
  clean self intersection components $K_{d_1},\dots, K_{d_l}$ with
  asymptotic winding numbers $w_1,\dots,w_l$, respectively, where
  $\dim(K_{d_j})=(n-d_j)$, $j=1,\dots,l$.
\end{itemize}

We trivialize $TX$ along parts of the map $f$ as follows.
\begin{itemize}
\item[$({\rm t1})$] Fix complex trivializations $Z_{\gamma}$ of the
  contact planes in the convex end of $X$ along all Reeb orbits
  $\gamma\in\{\gamma_1,\dots,\gamma_p\}$.
\item[$({\rm t2})$] Fix complex trivializations $Z_{\beta}$ of the
  contact planes in the concave end of $X$ along all Reeb orbits
  $\beta\in\{\beta_1,\dots,\beta_q\}$.
\end{itemize}
If $\alpha\in\{\gamma_1,\dots,\gamma_p\}$ or
$\alpha\in\{\beta_1,\dots,\beta_q\}$ then the linearized Reeb flow
induces a $1$-parameter family of symplectomorphisms $\Phi_t\colon
\xi_{\alpha(0)}\to\xi_{\alpha(t)}$, where $\xi_{\alpha(t)}$ is the
contact hyperplane at $\alpha(t)$, $t\in[0,T]$. Using the
trivialization $Z_\alpha$ from $({\rm t1})$ or $({\rm t2})$, we view
$\Phi_t$ as a path of symplectomorphisms
$\Phi^{Z_\alpha}_t\colon\C^{n-1}\to\C^{n-1}$. Write
\begin{equation}\label{eqn:CZ}
\mu_{\rm CZ}(\alpha,Z_\alpha)
\end{equation}
for the Conley-Zehnder index of the path $\Phi^{Z_\alpha}_t$, $0\le
t\le T$ (see~\cite{EES}).
\begin{remark}\label{r:CZind}[cf.~\cite{EES}]
The Conley-Zehnder index of a path $\Psi_t\colon \C^{m}\to\C^{m}$,
$0\le t\le 1$ is the Maslov index of the path of Lagrangian planes in
$\C^{m}\oplus\C^{m}$ corresponding to the graph of $\Psi_{t}$. The
Maslov index of a path $L_t$, $0\le t\le 1$, of Lagrangian planes in
$\C^{k}$ equals $\la\mu,[\hat L]\ra-\frac{k}{2}$, where $\mu$ is the
Maslov class and where $\hat L$ is the loop of Lagrangian planes
obtained by closing $L$ by a positive  rotation taking $L_1$ and
$L_0$. Here a positive rotation is defined as follows.  Two Lagrangian
subspaces $V_0$ and $V_1$ in $\C^{k}$ defines a decomposition
$W=W^1\oplus\dots\oplus W^r$ into orthogonal subspaces and a complex
angle $(\theta_1,\dots,\theta_r)$,
$0\le\theta_1<\theta_2<\dots<\theta_r<\pi$ as follows. Let $\theta_1$
be the smallest number in $[0,\pi)$ such that
\[
\dim\left((e^{i\theta_1}\cdot V_0)\cap V_1\right)\ge 1.
\]
Let $W^{1}\subset\C^{k}$ be the complex subspace generated by $e^{i\theta_1}\cdot V_0$ and let $W'$ be its orthogonal complement. Then $V_0'=W'\cap(e^{i\theta_1}\cdot V_0)$ and $V_1'=W'\cap V_1$ are Lagrangian subspaces.
Let $\theta_1'$ be the smallest number in $(0,\pi)$ such that
\[
\dim\left((e^{i\theta_1'}\cdot V_0')\cap V_1'\right)\ge 1.
\]
Let $\theta_2=\theta_1'+\theta_1$ and let $W_2\subset W'\subset W$ be the complex subspace generated by $e^{i\theta_1'}\cdot V_0'$. Repeating this construction we get a decompositon and complex angles as claimed.

The {\em positive rotation} taking $V_0$ to $V_1$ is the $1$-parameter family of linear transformations which acts by multiplication by $e^{i\theta_jt}$, $t\in[0,1]$ on $W^j$, $j=1,\dots,r$. The {\em negative rotation} taking $V_0$ to $V_1$ acts by multiplication by $e^{-i(\pi-\theta_j)t}$, $t\in[0,1]$ on $W^j$, $j=1,\dots,r$.
\end{remark}

\begin{itemize}
\item[$({\rm t3})$] Fix complex trivializations $Z_c$ of the
contact planes along all Reeb chords $c\in \{c_1,\dots,c_s\}$ of the Legendrian submanifold in the convex end which have the property that the linearized Reeb flow along the chord $c$ expressed in $Z_c$ is
constantly equal to the identity.
\item[$({\rm t4})$] Fix complex trivializations $Z_b$ of the
contact planes along all Reeb chords $b\in \{b_1,\dots,b_t\}$ of the Legendrian submanifold in the concave end which have the property that the linearized Reeb flow along the chord $b$ expressed in $Z_b$ is constantly equal to the identity.
\end{itemize}
Completing these trivializations with a vector field in the symplectization direction we get
trivializations of $TX$ along any Reeb orbit and along any Reeb
chord appearing as asymptotic data for $f$.
\begin{itemize}
\item[$({\rm t5})$] Fix complex trivializations $Z_C$ of $f^*TX$ along
each component $C$ of the complement of the punctures in $\pa S$ with
the following properties. If an endpoint of $C$ is a Reeb chord
puncture at a Reeb chord $a$ then $Z_C=Z_a$ at the corresponding Reeb chord endpoint in some neighborhood of the endpoint of $C$. If a Lagrangian intersection puncture is the common endpoint of
boundary components $C$ and $C'$ then $Z_C=Z_{C'}$ at the common endpoint.
\end{itemize}
The choices $({\rm t1})-({\rm t2})$ give trivializations of
$f^{\ast}TX$ near each interior puncture in $S$ and the choices $({\rm
  t3})-({\rm t5})$ give a trivializations $Z_{\pa_j f}$ of $f^{\ast}
TX$ along the $j^{\rm th}$ component $C_j$ of the boundary $\pa S$,
where we think of punctures as marked points so that $\pa S$ becomes a
closed $1$-manifold. Let
\begin{equation}\label{eqn:relc_1}
c_1^{\rm rel}\bigl(u^{\ast}(TX);Z_{\pa f};Z_{\gamma_1},\dots,Z_{\gamma_p};Z_{\beta_1},\dots,Z_{\beta_q}\bigr)
\end{equation}
denote the obstruction to extending this trivialization over $S$. Here we think of the obstruction as the number arising from evaluating the obstruction class on the orientation class of $(S_0,\pa S_0)$, where $S_0$ is the surface obtained from $S$ by removing small open disks around all its interior punctures and where the bundle is trivialized along $\pa S_0$.

Let $\Lambda$ denote a Legendrian submanifold at one of the ends of $X$ and let $a\in\{c_1,\dots,c_s\}$ or $a\in\{b_1,\dots,b_t\}$ be a Reeb chord of $\Lambda$. Let $a^{-}$ denote the endpoint of $a$ where the Reeb vector field points into $a$, and let $a^{+}$ denote the other endpoint of $a$. The image of the tangent space $T_{a^{-}}\Lambda$ under the linearized Reeb flow along $a$ is a Lagrangian plane $(T_{a^{-}}\Lambda)'\subset\xi_{a^{+}}$, where $\xi_y$ denotes the contact plane at $y$. Assume that $a$ is generic in the sense that the two Lagrangian subspaces $(T_{a^{-}}\Lambda)'$ and $T_{a^{+}}\Lambda$ of $\xi_{a^{+}}$ intersect transversely (after small perturbation, all Reeb chords are generic). Let
\[
R_{a^+}^{\rm neg}(a^-,a^+)\colon\xi_{a^{+}}\to\xi_{a^{+}}
\]
denote the rotation in $\xi_{a^{+}}$  in the negative direction which takes  $(T_{a^{-}}\Lambda)'$ to $T_{a^{+}}\Lambda$.  Let
\[
R_{a^{-}}^{\rm neg}(a^{+}, a^{-})\colon\xi_{a^{-}}\to\xi_{a^{-}}
\]
be defined similarly, rotating the image $(T_{a^{+}}\Lambda)'$ of $T_{a^{+}}\Lambda$, under the backwards linearized Reeb flow along $a$, in the negative direction in $\xi_{a^{-}}$ to $T_{a^{-}}\Lambda$.

Let $C_j'$ denote the complement of the punctures in the $j^{\rm th}$ component $C_j\subset \pa S$. Then the tangent planes to $L$ along $f(C'_j)$ expressed in the trivializations $Z_{\pa_j f}$ constitute a collection of paths of Lagrangian planes in $\C^{n}$. We close these paths to a loop as follows:
\begin{itemize}
\item[$({\rm t3}')$] The tangent planes of $L=\Lambda\times\R$ at
  endpoints of a Reeb chord $c\in\{c_1,\dots,c_s\}$ are connected by
  the product of the linearized Reeb flow along $c$ in $\xi$ and the
  identity in the symplectization direction, followed by the path
\[
R_{c^{+}}^{\rm
  neg}(c^{-},c^{+})\left((T_{c^{-}}\Lambda)'\right)\oplus\R\subset
\xi_{c^{+}}\oplus\C.
\]
\item[$({\rm t4'})$] The tangent planes of $L=\Lambda\times\R$ at
  endpoints of a Reeb chord $b\in\{b_1,\dots,b_t\}$ are connected by
  the backwards linearized Reeb flow along $b$ in $\xi$ and the
  identity in the symplectization direction, followed by the path
\[
R_{b^{-}}^{\rm
  neg}(b^{+},b^{-})\left((T_{b^{+}}\Lambda)'\right)\oplus\R\subset
\xi_{b^{-}}\oplus\C.
\]
\item[$({\rm t0'})$] The tangent planes at a Lagrangian intersection
  puncture mapping to $K_{d}\in\{K_{d_1},\dots,K_{d_l}\}$ of
  asymptotic winding number $w$ correspond to the planes $\R^{n-d}\times
  \R^{d}$ or $\R^{n-d}\times i\R^{d}$ in the coordinates $({\rm
    t0})$. Connect these planes by multiplying the tangent plane of
  the boundary component oriented toward the puncture with the matrix
\begin{equation}\label{eqn:Lagclose}
\left(
\begin{matrix}
1 & 0\\
0 &  e^{-sw\pi i}
\end{matrix}
\right),
\quad 0\le s\le 1,
\end{equation}
in the local $\C^{n-d}\times \C^{d}$-coordinates of $({\rm t0})$.
\end{itemize}
Define
\begin{equation}\label{eqn:Maslov}
\mu(\pa_j f,Z_{\pa_j f})
\end{equation}
as the Maslov index of the loop of Lagrangian subspaces in $\C^{n}$
which corresponds to the $j^{\rm th}$ boundary component of $S$ and
which is constructed by closing the paths of Lagrangian planes as
described in $({\rm t3'})$, $({\rm t4'})$, and $({\rm t0'})$.

Let $\MM(f)$ denote the moduli space of holomorphic curves in $X$ with
boundary on $L$, with punctures at Reeb orbits, Reeb chords, and at
Lagrangian self intersection components as described above,  which
have the same additional structure as $f$
(i.e. asymptotics and Lagrangian intersection punctures, including
asymptotic windings), have domain diffeomorphic to $(S,\pa S)$, and which are
homotopic to $f$ through (continuous) maps respecting the additional
structure. Recall that $\dim(X)=2n$, let $g$ denote the genus of $S$,
and let $r$ denote the number of boundary components of $\pa S$.

\begin{thm}\label{thm:dim}
With the notation from \eqref{eqn:CZ}, \eqref{eqn:relc_1}, and
\eqref{eqn:Maslov}, the formal dimension of $\MM(f)$ is given by
\begin{align*}
\dim\bigr(\MM(f)\bigr)&=(n-3)(2-2g-r)+(s+t+l)\\
&+\sum_{j=1}^{p}\Bigl(\mu_{\rm CZ}\bigl(\gamma_j,Z_{\gamma_j}\bigr)-(n-3)\Bigr)\\
&-\sum_{j=1}^{q}\Bigl(\mu_{\rm CZ}\bigl(\beta_j,Z_{\beta_j}\bigr)+(n-3)\Bigr)\\
&+\sum_{j=1}^{r}\mu\bigl(\pa_j f,Z_{\pa_j f}\bigr)\\
&+2c_1^{\rm rel}\bigl(u^{\ast}(TX);Z_{\pa
  f};Z_{\gamma_1},\dots,Z_{\gamma_p};Z_{\beta_1},\dots,Z_{\beta_q}\bigr).
\end{align*}
\end{thm}

\begin{remark}
As mentioned in Section \ref{sec:intro}, it follows from Theorem \ref{thm:dim} that the contribution from a Lagrangian intersection puncture mapping to a codimension $d$ clean intersection component with asymptotic winding number $w$ equals $1-wd$. Here $1$ is the contribution to $l$ and $-wd$ is the contribution to the Maslov index from the rotation in \eqref{eqn:Lagclose}.
\end{remark}

\begin{proof}
We consider first the case when there are no Lagrangian intersection punctures. The formal dimension of $\MM(f)$ equals the Fredholm index of the linearization of the $\bar\pa_J$-equation at $f$. The source space of this operator splits into the direct sum of an infinite dimensional functional analytic space of vector fields along $f$ and the tangent space of the space of conformal structures of the domain $S$ of $f$. We denote the restriction of the linearized $\bar\pa_J$-operator to the space of vector fields by $\bar\pa_{\rm vf}$.

The index of $\bar\pa_{\rm vf}$ remains constant as the operator is deformed through Fredholm operators. Consider the symplectization direction in $f^{\ast}TX$ near any boundary puncture in $S$. The boundary condition in this direction is degenerate. In order to describe a neighborhood of the map $f$ in a functional analytic setting (e.g. a polyfold neighborhood of $f$), one would use a Sobolev space with small positive exponential weights at the punctures and augment that space by one cut-off solution corresponding to translations in the $\R$-direction for each puncture. (With notation as above, if $\theta_{m}$ and $\theta_M$ denotes the smallest and largest complex angles respectively of $(T_{a^{-}}\Lambda)'$ and $T_{a^{+}}\Lambda$ over all Reeb chords $a\in\{c_1,\dots,c_p\}\cup\{b_1,\dots,b_t\}$ then the weight being small means that it is smaller than $\min\{\theta_m,\pi-\theta_M\}$.) However, for index purposes, this is equivalent to forgetting the cut-off solution and changing the weight to a small negative exponential weight. We will work in the setting of small negative exponential weight without auxiliary solutions below.

The first deformation of $\pa_{\rm vf}$ will change the
$\xi$-directions of the boundary condition near boundary punctures so
that they look like the symplectization direction. Consider the
boundary condition at a boundary puncture mapping to a Reeb chord
$a$. We deform it as follows. Rotate the image of the tangent space
$(T_{a^{\pm}}\Lambda)'$ at the endpoint $a^{\mp}$ of a boundary arc of
$f$ under the linearized Reeb flow along $a$ (forwards or backwards
according to the sign of the boundary puncture) in the negative
direction to $T_{a^{\mp}}\Lambda$ and simultaneously change the weight
at this puncture in the $\xi$-directions from its initial value
$1=e^{0}$ to a small {\em negative} exponential weight. It is
straightforward to check that this gives a path of Fredholm operators,
compare \cite[Proposition 6.14]{EES}. Denote the operator at the
endpoint of this path $\bar\pa_{\rm vf}'$. Consider the surface $\hat
S$ which is $S$ with boundary punctures erased. Since the change of
coordinates taking a neighborhood $[0,\infty)\times[0,1]$ of a
puncture in $S$ to a neighborhood of $0\in\H$ of the corresponding
point in $\hat S$ is $w\mapsto e^{-\pi w}$ it follows that the index
of the operator $\bar\pa_{\rm vf}'$ on $S$ equals the index of the
$\bar\pa$-operator on the surface $\hat S$,  with boundary condition
naturally induced from the boundary condition of $\bar\pa_{\rm vf}'$,
see \cite[Proposition 6.13]{EES}. We denote this operator on $\hat
S$ by $\bar\pa_{\hat S}$. By definition, in the trivialization
$Z_{(\pa_j f)}$ along the $j^{\rm th}$ boundary component of $\hat S$,
the Lagrangian boundary condition of $\bar\pa_{\hat S}$ has Maslov
index $\mu\bigl(\pa_j f;Z_{\pa_jf}\bigr)$.

We next consider interior punctures. Also here, the asymptotic
operator is degenerate in the symplectization direction. As in the
case of boundary punctures, one would use small positive exponential
weights and cut-off solutions to define functional neighborhoods, but
in order to compute the index we might as well use small negative
exponential weights and no cut-off solutions to compute the
index. (Here small refers to small when compared to the distance
between the eigenvalues of the linearized return maps and $1$.)

Fix capping spheres of all Reeb orbits at interior punctures. A
capping sphere of a positive (negative) puncture where $f$ is
asymptotic to a Reeb orbit $\alpha$ is a once punctured sphere with a
trivial $\C^{n-1}\oplus\C$-bundle over it with trivialization which
extends the trivialization $Z_\alpha$ given near the puncture and with
a $\bar\pa$-operator with the asymptotics of a negative (positive)
puncture at $\alpha$ in the $\C^{n-1}$-direction and with trivial
asymptotics and small positive exponential weight in the
symplectization direction corresponding to $\C$. Thus, the capping
operator $\bar\pa^{+}_{\alpha}$ of a positive puncture at $\alpha$ has
index
\[
\ind(\bar\pa^+_\alpha)=(n-1)-\mu_{\rm CZ}(\alpha,Z_\alpha),
\]
and the capping operator $\bar\pa^{-}_{\alpha}$ of a negative puncture
at $\alpha$ has index
\[
\ind(\bar\pa^-_\alpha)=(n-1)+\mu_{\rm CZ}(\alpha,Z_{\alpha}),
\]
see \cite{Sw, BM}.
A well-known argument shows that the index is additive under linear
gluing of operators. We make one remark concerning this result in the
present setup: the symplectization $\C$-component of the operator on
a gluing neck limits to the standard operator on the infinite cylinder
with positive exponential weight at one end and negative exponential
weight at the other. This Fredholm operator is invertible and the
usual linear gluing argument applies.

The result of gluing the capping spheres at the punctures of $\hat S$
and the capping operators to the operator $\bar\pa_{\hat S}$ is a
$\bar\pa$-operator $\bar\pa_{\bar S}$ on a surface $\bar S$ of genus
$g$ with $r$ boundary components and a Lagrangian boundary condition
along each boundary component. The complex bundle over $\bar S$ comes
equipped with a trivialization $Z$ near its boundary. Let $\mu(\pa\bar
S,Z)$ denote the total Maslov index of the Lagrangian boundary
condition of the boundary measured with respect to the trivialization
$Z$ and let $c_1^{\rm rel}(Z)$ denote the relative Chern class which
is the obstruction to extending $Z$ from $\pa\bar S$ to all of $\bar
S$. Doubling $\bar S$ as well as the operator $\bar\pa_{\bar S}$ over
the boundary of $\bar S$ and applying the Riemann-Roch formula in
combination with complex conjugation gives
\[
\ind(\bar\pa_{\bar S})= n(2-2g-r) + \mu(\pa\bar S,Z) + 2c_1^{\rm rel}(Z).
\]
Additivity of the index then gives
\begin{align*}
\ind(\bar\pa_{\rm vf})&=\ind(\bar\pa_{\hat S})\\
&=n(2-2g-r)\\
&\ \ +\sum_{j=1}^{p}\bigl(\mu_{\rm CZ}(\gamma_j,Z_{\gamma_j})-(n-1)\bigr)
-\sum_{j=1}^{q}\bigl(\mu_{\rm CZ}(\beta_j,Z_{\beta_j})+(n-1)\bigr)\\
&\ \ +  \mu(\pa\bar S,Z) + 2c_1^{\rm rel}(Z)\\
&=n(2-2g-r)\\
&\ \ +\sum_{j=1}^{p}\bigl(\mu_{\rm CZ}(\gamma_j,Z_{\gamma_j})-(n-1)\bigr)
-\sum_{j=1}^{q}\bigl(\mu_{\rm CZ}(\beta_j,Z_{\beta_j})+(n-1)\bigr)\\
&\ \ +\sum_{j=1}^{r}\mu\bigl(\pa_j f, Z_{\pa_j f}\bigr)\\
&\ \ +2c_1^{\rm rel}\bigl(f^{\ast}(TX);Z_{\pa
  f};Z_{\gamma_1},\dots,Z_{\gamma_p};Z_{\beta_1},\dots,Z_{\beta_q}\bigr).
\end{align*}

In order to compute the dimension it remains only to compute the
dimension $\dim(\TT)$ of the space $\TT$ of conformal structures on
$S$. Doubling a surface with $r$ boundary components in a similar way
as above, studying the $\bar\pa$-equation for vector fields along the
surface which are tangent to the boundary along the boundary, and
noting that each interior puncture adds $2$ degrees of freedom and
each boundary puncture adds $1$ degree of freedom, we find that the
dimension of the space of conformal structures on $S$ equals
\[
\dim(\TT) = 3r+(s+t)+2(p+q)-6+6g.
\]
We thus have
\begin{align*}
\dim\bigl(\MM(f)\bigr)&=\ind(\bar\pa_{\rm vf})+\dim(\TT)\\
&=(n-3)(2-2g-r)+(s+t)\\
&\ \ +\sum_{j=1}^{p}\bigl(\mu_{\rm CZ}(\gamma_j,Z_{\gamma_j})-(n-3)\bigr)
-\sum_{j=1}^{q}\bigl(\mu_{\rm CZ}(\beta_j,Z_{\beta_j})+(n-3)\bigr)\\
&\ \ +\sum_{j=1}^{r}\mu\bigl(\pa_j f, Z_{\pa_j f}\bigr)\\
&\ \ +2c_1^{\rm rel}\bigl(f^{\ast}(TX);Z_{\pa
  f};Z_{\gamma_1},\dots,Z_{\gamma_p};Z_{\beta_1},\dots,Z_{\beta_q}\bigr),
\end{align*}
finishing the proof in the case when there are no Lagrangian
intersection punctures.

Consider next the case when there are Lagrangian intersection
punctures. In order to define a functional analytic neighborhood of a
map $f$ with such punctures of given asymptotic winding number $w$ mapping to
a codimension $d$ component of the clean intersection, we puncture the
boundary of $S$ and identify a neighborhood of the puncture in the
domain with $[0,\infty)\times[0,1]$ by the change of variables
$z=e^{-\pi w}$, $w=\tau+it\in [0,\infty)\times[0,1]$. Then the Taylor
expansion \eqref{eqn:asymptwexp} gives
%\[
%\bigl(f_1(\tau+it),f_2(\tau+it)\bigr)=\left(\sum_{k=0} a_k
%e^{-k\pi(\tau+it)},\,\,\,\sum_{k=w}^{\infty} c_k e^{-\pi k(\tau+i
%t)}\right)
%\]
\[
   f_1(\tau+it) = \sum_{j=0}^\infty a_j e^{-\pi j(\tau+it)},\qquad
   f_2(\tau+it) = \sum_{j=0}^\infty c_j e^{-\pi(j+w)(\tau+it)}.
\]
It follows that a neighborhood can be modeled on a Sobolev space with
positive exponential weight $e^{(w-\frac{1}{100})\tau}$ augmented by
the space of cut off solutions spanned by
\[
\psi\cdot a_0,\,\, \psi\cdot a_1e^{-\pi(\tau+it)},\,\,\dots\,\,,\,\,
\psi\cdot a_v e^{-\pi v(\tau+it)},
\]
where $\psi$ is a cut off function on $[0,\infty)\times[0,1]$, where
$v$ is the largest integer smaller than $w-\frac{1}{100}$, and where
$a_j\in\R^{n-d}$. (This  augmentation space has dimension
$(n-d)(v+1)$.)

At Lagrangian intersection punctures where the map switches local
sheets of $L$, we observe that, as for Reeb chords above, we can close
up the boundary condition at a Lagrangian intersection puncture by
rotating $-\frac{\pi}{2}$ in $\C^{d}$, keeping the weight, and obtain
a family of Fredholm operators.

Finally, we interpret the above weights in terms of the closed up
boundary condition along $\pa \hat S$. In the source we use the change
of variables $z=e^{-\pi w}$, $w=\tau+it\in S$ and $z\in \hat S$
as above. We conclude that an exponential weight in
$w$-coordinates of magnitude $k'\pi$, where $k-1 <k'<k$ for an integer
$k\ge 1$, corresponds to the condition that the sections of $f^{\ast}TX$
and their first $k-1$ derivatives vanishes at $0$ in the
$z$-coordinates. Thus the dimension formula is obtained by applying
the formula above to the boundary condition obtained by closing up the
Lagrangian boundary conditions at each Lagrangian intersection
puncture with a minimal negative rotation (i.e., if the map switches
sheets at the puncture we rotate by $-\frac{\pi}{2}$ in the
$C^{d}$-factor complementary to $T^{\ast}K$ and by $0$ in the
$\C^{n-d}$-factor corresponding to $T^{\ast}K$, and if the map does
not switch sheets we rotate by $0$ in both factors) and adding
\[
1-w'd
\]
for each Lagrangian intersection puncture, where $w'$ is the largest
integer smaller than $w$. Here $1$ comes from the increase in the
dimension of the space of conformal structures $\TT$. The theorem then
follows by definition of the close up at Lagrangian intersection
punctures, see \eqref{eqn:Lagclose}.
\end{proof}

\end{document}